\documentclass[hidelinks,onefignum,onetabnum]{siamart220329}

\usepackage{lipsum}
\usepackage{amsfonts,amssymb,tikz,amsmath,subfigure}
\usepackage{graphicx}
\usepackage{epstopdf}
\usepackage[all,import]{xy}
\usepackage{multirow}
\usepackage{tikz-cd}
\usetikzlibrary{matrix, calc, arrows}
\usepackage[normalem]{ulem}
\usepackage{smartdiagram}
\useunder{\uline}{\ul}{}

\usepackage{algorithmic}
\ifpdf
  \DeclareGraphicsExtensions{.eps,.pdf,.png,.jpg}
\else
  \DeclareGraphicsExtensions{.eps}
\fi

\usetikzlibrary{arrows.meta}
\usetikzlibrary{er,positioning}

\renewcommand{\vec}[1]{\mbox{\boldmath \small $#1$}}

\usepackage{graphicx}

\usepackage{latexsym}
\usepackage{bm}
\usepackage{amsmath}
\usepackage{amssymb}
\usepackage{color,tikz}
\usetikzlibrary{matrix}
\usepackage{cases}
\usetikzlibrary{shapes,arrows}
\usepackage{makecell}
\usepackage{multicol, multirow}
\usepackage{epsfig}
\usepackage{hyperref}
\usepackage{graphicx}
\usepackage{euscript}
\usepackage{epsfig}
\usepackage{subfigure} 
\usepackage{caption}

\renewcommand{\vec}[1]{\mbox{\boldmath$#1$}}
\newcommand{\vol}{\ensuremath{\mathrm{vol}}}
\newcommand{\sign}{\ensuremath{\mathrm{sign}}}
\newcommand{\sgn}{\ensuremath{\mathrm{Sgn}}}
\newcommand{\median}{\ensuremath{\mathrm{median}}}
\newcommand{\set}[1]{\{ #1\}}
\newcommand{\dif}{\ensuremath{\mathrm{d}}}

\DeclareMathOperator*{\argmax}{\ensuremath{\mathrm{argmax\,}}}
\DeclareMathOperator*{\argmin}{\ensuremath{\mathrm{argmin\,}}}
\DeclareMathOperator*{\nen}{\ensuremath{\mathrm{NEN\,}}}
\DeclareMathOperator*{\tc}{\ensuremath{\mathrm{TC\,}}}

\newsiamremark{remark}{Remark}
\newsiamremark{hypothesis}{Hypothesis}
\crefname{hypothesis}{Hypothesis}{Hypotheses}
\newsiamthm{claim}{Claim}
\newsiamremark{defn}{Definition}
\newsiamremark{com}{Note}
\newsiamremark{cor}{Corollary}
\newsiamremark{pro}{Proposition}
\newsiamremark{example}{Example}

\headers{A simple inverse power method for balanced graph cut}{S. Shao and C. Yang}

\title{A simple inverse power method for balanced graph cut\thanks{Submitted to the editors on \today.
\funding{This work was funded by the National Key R \& D Program of China (No. 2022YFA1005102) and
the National Natural Science Foundation of China (Nos. 12325112, 12288101).}}}

\author{Sihong Shao\thanks{CAPT, LMAM and School of Mathematical Sciences, Peking University, Beijing 100871, China 
  (\email{sihong@math.pku.edu.cn}). 
  }
\and Chuan Yang\thanks{School of Mathematical Sciences, Peking University, Beijing 100871, China
  (\email{chuanyang@pku.edu.cn}).
  }
}

\usepackage{amsopn}

\ifpdf
\hypersetup{
	pdftitle={A simple inverse power method for balanced graph cut},
	pdfauthor={S. Shao and C. Yang}
}
\fi

\begin{document}
	
	\maketitle
	
	\begin{abstract}
		The existing inverse power ($\mathbf{IP}$) method for solving the balanced graph cut lacks local convergence and its inner subproblem requires a nonsmooth convex solver. To address these issues, we develop a simple inverse power ($\mathbf{SIP}$) method using a novel equivalent continuous formulation of the balanced graph cut, and its inner subproblem allows an explicit analytic solution, which is the biggest advantage over $\mathbf{IP}$ and constitutes the main reason why we call it \emph{simple}. By fully exploiting the closed-form of the inner subproblem solution, we design a boundary-detected subgradient selection with which $\mathbf{SIP}$ is proved to be locally converged. We show that $\mathbf{SIP}$ is also applicable to a new ternary valued $\theta$-balanced cut which reduces to the balanced cut when $\theta=1$. When $\mathbf{SIP}$ reaches its local optimum, we seamlessly transfer to solve the $\theta$-balanced cut within exactly the same iteration algorithm framework and thus obtain $\mathbf{SIP}$-$\mathbf{perturb}$ --- an efficient local breakout improvement of $\mathbf{SIP}$, which transforms some ``partitioned" vertices back to the ``un-partitioned" ones through the adjustable $\theta$. Numerical experiments on G-set for Cheeger cut and Sparsest cut demonstrate that $\mathbf{SIP}$ is significantly faster than $\mathbf{IP}$ while maintaining approximate solutions of comparable quality, and $\mathbf{SIP}$-$\mathbf{perturb}$ outperforms \texttt{Gurobi} in terms of both computational cost and solution quality.
	\end{abstract}
	
	
	
	\section{Introduction}
	\label{sec:intro}

	We consider the following balanced cut problem defined on an undirected graph $G=(V,E,w,\mu)$ 
	\begin{equation}
		\label{prob:balance}
		h(G)=\min\limits_{S\subset V,S\not\in\{\emptyset,V\}} \frac{|E( S,S^c)|}{\min\{\vol(S),\vol(S^c)\}},
	\end{equation}
	where $V=\{1, 2, \dots, n\}$ gives the vertex set, 
	$\mu=(\mu_1, \mu_2, \dots, \mu_n)$ with $\mu_i$ being a positive weight on $i\in V$,
	$\vol(S)=\sum_{i\in S}\mu_i$ calculates the volume of $S\subset V$,
	$E$ is the edge set, 
	$W=(w_{ij})_{n\times n}$  with $w_{ij}$ being a positive weight on $\{i,j\}\in E$ and vanishing on $\{i,j\}\notin E$,
	$E(S,T)$ collects the set of the edges across $S\subset V$ and $T\subset V$ with the amount of $|E(S,T)|=\sum_{\{i,j\}\in E(S,T)}w_{ij}$,
	$S^c=V\backslash S$ denotes the complement of $S$ in $V$ and $(S,S^c)$ presents a binary cut of $G$.
	Two typical examples of Eq.~\eqref{prob:balance} are Cheeger cut where $\mu_i = d_i$ holds for $\forall\, i\in V$ with $d_i=\sum_{j:\{i,j\}\in E}w_{ij}$ being the degree of the $i$-th vertex \cite{Alon1986}, and Sparest cut where $\mu_i = 1$ holds for $\forall\, i\in V$ \cite{arora2009expander}.  Both problems are NP-hard \cite{hochbaum2013polynomial} and have various real-world applications in the fields of clustering \cite{bresson2013multiclass,szlam2010total,HeinBuhler2010}, community detection \cite{van2016local}, and computer vision \cite{shi2000normalized}. There exist multiple heuristic algorithms that manage to produce high-quality solutions within reasonable computational time, including the max-flow quotient-cut improvement algorithm \cite{lang2004flow}, the breakout tabu search \cite{lu2019stagnation} and the hybrid evolutionary algorithm \cite{lu2020hybrid}. Recently, developing efficient algorithms with continuous flavor, as an alternative, has also attracted more and more attention \cite{HeinBuhler2010,szlam2010total,Chang2015} and this work falls into such category.

	The first well-known continuous equivalent formulation of the balanced cut problem~\eqref{prob:balance} reads \cite{Chung1997,chang2021lovasz}
	\begin{equation}
		\label{conti-prob:balance}
		h(G)=\min\limits_{\text{nonconstant~}\vec x\in\mathbb{R}^n}\frac{I(\vec x)}{N(\vec x)},
	\end{equation}
	where
	\[
	I(\vec x) =\sum_{\{i,j\}\in E}w_{ij}|x_i-x_j|, \quad 
	N(\vec x) =\min_{c\in\mathbb{R}}\sum_{i\in V}\mu_i|x_i-c|.
	\]
	Formally applying a Lagrange multiplier technique equipped with the subgradient into the above continuous equivalent formulation yields the so-called $1$-Laplacian eigenproblem and the resulting second smallest eigenvalue happens to be $h(G)$ \cite{Chang2015}.
	Accordingly, a three-step inverse power ($\mathbf{IP}$) method starting from an initial data $\vec x^0\in\mathbb{R}^n$ \cite{HeinBuhler2010,Chang2015}
	\begin{subequations}
		\label{iter1-cheeger}
		\begin{numcases}{}
			\vec x^{k+1}=\mathop{\argmin}\limits_{\|\vec x\|_1\leq 1} \{I(\vec x) - r^k\langle\vec x,\vec v^k\rangle\}, 
			\label{eq:twostep_x2}
			\\
			r^{k+1}=I(\vec x^{k+1})/N(\vec x^{k+1}),
			\label{eq:twostep_r2}
			\\
			\vec v^{k+1}\in\partial N(\vec x^{k+1}),
			\label{eq:twostep_s2}
		\end{numcases}
	\end{subequations}
	was proposed to approximate $h(G)$ after introducing a subgradient relaxation to the equivalent two-step Dinkelbach iteration for Eq.~\eqref{conti-prob:balance} \cite{dinkelbach1967nonlinear}. Here $\langle\cdot,\cdot\rangle$ denotes the standard inner product in $\mathbb{R}^n$, $\|\vec x\|_1 = \sum_{i\in V}|x_i|$ and ``$\partial$" calculates the subgradient. 
	However, solving the inner subproblem \eqref{eq:twostep_x2} requires additional computational burden and potentially results in time-consuming calculations at each iteration, thereby bothering the researchers to design extra algorithms \cite{hein2011beyond} or to use extra optimization solvers (e.g. CVX \cite{Chang2015}). To compensate this, we first propose an alternative equivalent continuous formulation for the balanced cut problem~\eqref{prob:balance}
	\begin{equation}
		\label{conti2-prob:balance}
		h(G)=\min\limits_{\text{nonconstant~}\vec x\in\mathbb{R}^n}\frac{e\|\vec x\|_{\infty}-I^+(\vec x)}{N(\vec x)},
	\end{equation}
	and then obtain a new three-step iteration for approximating $h(G)$
	\begin{subequations}
		\label{iter1-cheeger_2}
		\begin{numcases}{}
			\vec x^{k+1}=\mathop{\argmin}\limits_{\|\vec x\|_1 \le 1} \{\|\vec x\|_{\infty} - \langle\vec x,\vec s^k\rangle\}, 
			\label{eq:twostep_x2_2}
			\\
			r^{k+1}= B(\vec x^{k+1}), 
			\label{eq:twostep_r2_2}
			\\
			\vec s^{k+1}\in\partial H_{r^{k+1}}(\vec x^{k+1}),
			\label{eq:twostep_s2_2}
		\end{numcases}
	\end{subequations}
	where $e$ counts all the degrees in $V$, $\|\vec x\|_{\infty}=\max\{|x_1|,\ldots,|x_n|\}$, 
	$B(\vec x)$ denotes the objective function of Eq.~\eqref{conti2-prob:balance}, and 
	\begin{equation}
		\label{eq:Hr}
		I^+(\vec x)=\sum_{\{i,j\}\in E}w_{ij}|x_i+x_j|, \quad H_r(\vec x)=\frac{1}{e} (I^+(\vec x)+rN(\vec x))\text{~for~} r\ge 0.
	\end{equation}
	Compared with Eq.~\eqref{eq:twostep_x2} in $\mathbf{IP}$, the inner subproblem \eqref{eq:twostep_x2_2} has an explicit analytic solution \cite{SZZ2018}.
	That is, no any extra solver is needed in the three-step inverse power iteration \eqref{iter1-cheeger_2} and thus
	we name it the simple inverse power ($\mathbf{SIP}$) method. 
	In particular, we will prove that the objective function values $\{r^k\}$ of $\mathbf{SIP}$ strictly decrease to a local optimum within finite iterations when the subgradients in Eq.~\eqref{eq:twostep_s2_2} are chosen in a boundary-detected
manner.  Numerical experiments demonstrate that $\mathbf{SIP}$ outperforms $\mathbf{IP}$ in both solution quality and computing time 	 
	on the majority of graphs in G-set. This constitutes the first contribution of this work.

	In order to further improve the solution quality, we propose a ternary-valued generalization of Eq.~\eqref{prob:balance} --- 
	the $\theta$-balanced cut problem for $\theta\in[0,1]$, 
	\begin{equation}
		\label{prob:theta-balance}
		h_{\theta}(G)=\min\limits_{(V_1,V_2)\in \tc(V)} \frac{\theta\sum_{i\in (V_1\cup V_2)^c}d_i+2|E( V_1,V_2)|}{\sum\limits_{X\in\{V_1,V_2\}}\min\limits_{Y\in\{X,X^c\}}\vol(Y)},
	\end{equation}
	where 
	\begin{equation}
		\label{ternary-cut}
		\tc(V):=\left\{(V_1,V_2) \big| V_1\cap V_2=\emptyset,\,\emptyset\neq V_1\cup V_2\subseteq V,\,V_1^c\neq\emptyset,\, V_2^c\neq\emptyset\right\}
	\end{equation}
	gives the set of ``partitioned" vertex subset pairs, 
	and the  ``un-partitioned" part behind each pair is the complement of their union in $V$.
	Obviously, $h_{\theta}(G)$ possesses the same objective function as $h(G)$ if $V_1\cup V_2=V$. 
	More importantly, following the same way from the discrete form \eqref{prob:balance} of $h(G)$ to the continuous form \eqref{conti2-prob:balance} and to the approximate iteration form \eqref{iter1-cheeger_2}, we are still able to obtain an equivalent continuous formulation for the $\theta$-balanced cut,  
	\begin{equation}
		\label{conti-theta:balance}
		h_{\theta}(G) =\min\limits_{\text{nonconstant~}\vec x\in\mathbb{R}^n}\frac{\theta e\|\vec x\|_{\infty}+(1-\theta)\|\vec x\|_{1,d}-I^+(\vec x)}{N(\vec x)},
	\end{equation}
	and thus a simple inverse power method for approximating $h_{\theta}(G)$, denoted as $\mathbf{SIP}_{\theta}$ (see Eq.~\eqref{iter2-cheeger_2}). Here $\|\vec x\|_{1,d}=\sum_{i\in V}d_i|x_i|$.  That is, within exactly the same iteration algorithm framework, we may approximate $h(G)$ and $h_{\theta}(G)$ via $\mathbf{SIP}$ and $\mathbf{SIP}_{\theta}$, respectively, while sharing the same objective function in the subgradient selection phase (see Eqs.~\eqref{eq:twostep_s2_2} and \eqref{eq2:twostep_s2_2}). 
	Actually, Eq.~\eqref{conti-theta:balance} only introduces an extra $\theta$-dependent term in the objective function, $\theta e\|\vec x\|_{\infty}+(1-\theta)\|\vec x\|_{1,d}$, the correspondence of which in Eq.~\eqref{conti2-prob:balance} is $e\|\vec x\|_{\infty}$, 
	and thus using the fact $\forall\, \vec x\in\mathbb{R}^n, e\|\vec x\|_{\infty} \ge \|\vec x\|_{1,d}$ we have that $h_{\theta}(G)$ is gradually raised from $0$ to $h(G)$ as $\theta$ increases from $0$ to $1$. Namely, $h_0(G)\equiv 0$ and $h_1(G)\equiv h(G)$. 
	In this regard, we may see $\mathbf{SIP}_{\theta}$ as a theta method dynamics in a similar constructive manner to strengthen $\mathbf{SIP}$ without extra excessive calculations, thus revolved operations of $\mathbf{SIP}$ and $\mathbf{SIP}_{\theta}$ may achieve an improved simple inverse power perturbation method for approximating $h(G)$, dubbed $\mathbf{SIP}$-$\mathbf{perturb}$, the flowchart of which is shown in Fig.~\ref{flowchart}. When $\mathbf{SIP}$ fails to decrease the objective function values, some ``partitioned" vertices in its binary-valued output are transformed back to the ``un-partitioned" ones by $\mathbf{SIP}_{\theta}$, and thus we may expect that $\mathbf{SIP}$-$\mathbf{perturb}$ is an efficient local breakout improvement of $\mathbf{SIP}$. Here the parameter $\theta$ is used for allocating proportion of the ``un-partitioned" vertices in $V$: A smaller $\theta$ yields a larger difference between the current local-optimal binary cut and the undergoing ternary one. Our numerical results show that $\mathbf{SIP}$-$\mathbf{perturb}$ provides higher-quality solutions within a relatively short period of time when compared to the well-known solver \texttt{Gurobi} on  almost all graphs in G-set,
	thereby demonstrating its expansion of the search area as well as its enhancement to the diversification over $\mathbf{SIP}$.
	This constitutes the second contribution of this work.

	\begin{figure}[htbp]
		\centering
		\tikzstyle{pt} = [point]
		\tikzstyle{decision} = [diamond, draw, text width=4em, aspect=1.2, text badly centered, node distance=3cm, inner sep=0pt]
		\tikzstyle{decision2} = [diamond, draw, text width=6em, aspect=1.2, text badly centered, node distance=3cm, inner sep=0pt]
		\tikzstyle{block} = [rectangle, draw, 
		text width=8em, text centered, rounded corners, minimum height=4em]
		\tikzstyle{io} = [rectangle, draw, 
		text width=4em, text centered, rounded corners, minimum height=2em]
		\tikzstyle{line} = [draw, thick, -latex']
		\tikzstyle{cloud} = [draw, ellipse,fill=red!20, node distance=2.5cm,
		minimum height=2em]
		\tikzstyle{every node}=[scale=0.8]
		\begin{tikzpicture}
			\node [io] (init) {Start};
			
			\node [block, below of=init, node distance=1.8cm] (itersip) {$\mathbf{SIP}$: approximate $h(G)$};
			\node [decision2,  below of=itersip,node distance=2.5cm] (keep_1) {Is $B(\vec x)$  unchanged?};
			\coordinate [right of=keep_1, node distance=3cm] (pt1);
			\node [decision, below of=keep_1, node distance=2.7cm] (ext1) {Exceed $T_{\theta}$?};
			\node [io, right of=ext1,node distance=3cm] (ending) {Stop};
			\node [block, below of=ext1, node distance=2.1cm] (itertheta) {$\mathbf{SIP}_{\theta}$: approximate $h_{\theta}(G)$};
			\node [decision2,  below of=itertheta,node distance=2.5cm] (keep_2) {Is $T(\vec x)$  unchanged?};
			\coordinate [right of=keep_2, node distance=3cm] (pt2);
			\coordinate [left of=keep_2, node distance=3cm] (pt3);
			
			\path [line](init)--(itersip);
			\path [line] (itersip)--(keep_1);
			\draw [thick](keep_1) -- node [above, color=black] {No}(pt1);
			\path [line] (pt1)|-(itersip);
			\path [line] (keep_1)--node [anchor=east, color=black] {Yes} (ext1);
			\path [line] (ext1)-- node[above] {Yes} (ending);
			\path [line] (ext1)--node [anchor=east, color=black] {No} (itertheta);
			\path [line] (itertheta)--(keep_2);
			\draw [thick](keep_2) -- node [above, color=black] {No}(pt2);
			\path [line] (pt2)|-(itertheta);
			\draw [thick](keep_2) -- node [above, color=black] {Yes}(pt3);
			\path [line] (pt3)|-(itersip);

			
		\end{tikzpicture}
		\captionsetup{singlelinecheck=off}
		\caption[.]{\small The flowchart of $\mathbf{SIP}$-$\mathbf{perturb}$ --- an efficient local breakout improvement of $\mathbf{SIP}$ for the balanced cut problem~\eqref{prob:balance}. $\mathbf{SIP}$-$\mathbf{perturb}$ takes turns applying two iterative schemes $\mathbf{SIP}$ \eqref{iter1-cheeger_2} and $\mathbf{SIP}_{\theta}$ \eqref{iter2-cheeger_2}, where $\mathbf{SIP}_{\theta}$ serves as the local breakout technique for $\mathbf{SIP}$. $B(\vec x)$ and $T(\vec x)$ denote the objective functions of the equivalent continuous formulations for $h(G)$ \eqref{conti2-prob:balance} and $h_{\theta}(G)$ \eqref{conti-theta:balance}, respectively. $T_{\theta}$ gives the total rounds of applying $\mathbf{SIP}_{\theta}$. 
			That is,
			\begin{align*}
				\mathbf{SIP}\text{-}\mathbf{perturb}=\underbrace{(\mathbf{SIP}\rightarrow\mathbf{SIP}_{\theta})\rightarrow\cdots\rightarrow(\mathbf{SIP}\rightarrow\mathbf{SIP}_{\theta})}_{T_{\theta}\text{ rounds}}\rightarrow\mathbf{SIP}.
			\end{align*}
			Finally, $\mathbf{SIP}$-$\mathbf{perturb}$ outputs the best solution generated from all $T_{\theta}+1$ rounds of $\mathbf{SIP}$.} \label{flowchart}
	\end{figure}
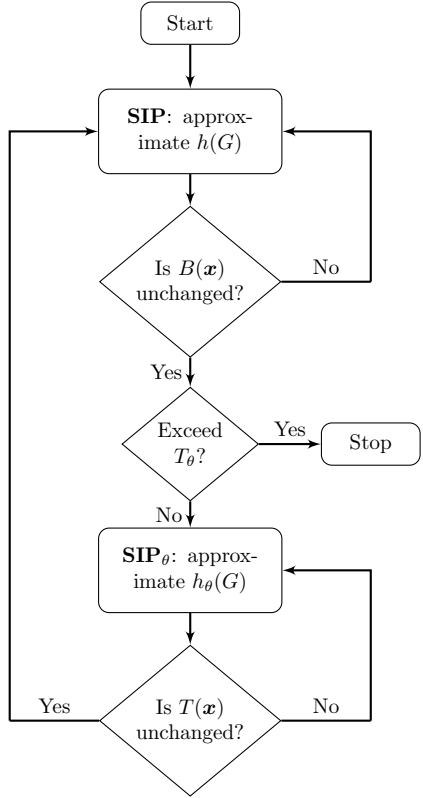

	The rest of paper is organized as follows. Section~\ref{sec:prob} proves the equivalent continuous formulations: Eq.~\eqref{conti2-prob:balance} for $h(G)$ and  Eq.~\eqref{conti-theta:balance} for $h_{\theta}(G)$, as well as their Dinkelbach iterative schemes. 
	In Section~\ref{sec:algorithm}, $\mathbf{SIP}$ and $\mathbf{SIP}_{\theta}$ are detailed with the analytical solution to the inner subproblem, carefully designed subgradient selection, and 	the local convergence analysis. Section~\ref{sec:experiments} conducts numerical experiments on G-set,  and Section~\ref{sec:conclusion} rounds the paper off with our conclusions.
	Throughout this paper, we use a bold black lowercase letter to denote a vector,  say $\vec x$, $\vec s$, $\vec b$, and write its $j$-th component as $x_j$, $s_j$, $b_j$, i.e., the corresponding lowercase letter in regular style equipped with the subscript $j$. When entering into an iterative scheme, we will add a superscript, say $k$, to specify that the vector $\vec x^k$ or its $j$-th component $x_j^k$ is at the $k$-th iteration.

	\section{Equivalent continuous formulations}
	\label{sec:prob}
	For any nonempty subset $S\subset V$, we define an indicative vector $\vec 1_S$ that each component $(\vec 1_S)_i$ equals to $1$ if $i\in S$ and $0$ if $i\in S^c$, thus the ternary vector $\vec 1_{V_1,V_2}=\vec 1_{V_1}-\vec 1_{V_2}\in\{-1,0,1\}^n$ satisfying $(V_1,V_2)\in\tc(V)$ (see Eq.~\eqref{ternary-cut}) is the vector-based representation of a nonconstant ternary cut composed of the ``partitioned" parts $V_1$, $V_2$ and the ``un-partitioned" part $(V_1\cup V_2)^c$, where ``ternary" may degenerate into ``binary" in certain cases, such as $V_1\cup V_2=V$, and ``nonconstant" signifies that ``ternary" can not reduce to ``unary"
	and not all the components of $\vec 1_{V_1,V_2}$ are equal. Note in passing that the nonconstant constraint did not appear in the discussion for maxcut \cite{SZZ2018} since the trivial zero cut value reaches at a constant vector. 


	We first prove the  equivalent continuous formulations:
	Eq.~\eqref{conti2-prob:balance} for $h(G)$ and Eq.~\eqref{conti-theta:balance} for $h_{\theta}(G)$. 
	Before that, we need a lemma related to the set-pair Lov{\'a}sz extension of a given set-pair function $f: \mathcal{P}_2(V)\rightarrow[0,+\infty)$
	with $\mathcal{P}_2(V)=\{(V_1, V_2): V_1, V_2 \subset V, V_1 \cap V_2=\varnothing\}$, denoted by $f^L$, which reads \cite{chang2021lovasz}
	\begin{equation}
		\label{eq:lovasz}
		f^L(\vec{x})=\int_0^{\|\vec x\|_{\infty}} f\left(V_t^{+}(\vec{x}), V_t^{-}(\vec{x})\right) \dif t,
	\end{equation}
	where $V_t^{ \pm}(\vec{x})=\left\{i \in V: \pm x_i>t\right\}$.

	\begin{lemma}\label{lem:subgradient}
		For any $\vec x\in\mathbb{R}^n$ and its induced ternary vector $\vec y=\vec 1_{V_0^+(\vec x),V_0^-(\vec x)}$, $\partial f^L(\vec x)\subseteq\partial f^L(\vec y)$ holds for any convex set-pair Lov{\'a}sz extension $f^L$ given in Eq.~\eqref{eq:lovasz}.
	\end{lemma}
	
	\begin{proof}
		Let $\vec u=M\vec x-\vec y$ with the constant $M$ satisfying $\min\limits_{i\in V_0^+(\vec x)\cup V_0^-(\vec x)}\left\{M|x_i|\right\}=2$. Then for each $i\in V$, $u_i>0$ if $y_i>0$, $u_i<0$ if $y_i<0$, and $u_i=0$ if $y_i=0$. According to Proposition 2.5 of \cite{chang2021lovasz}, for any $\vec v\in\partial f^L(\vec x)$, we have 			
		$$
		\langle\vec y,\vec v\rangle\leq f^L(\vec y)=f^L(\vec u+\vec y)-f^L(\vec u)\leq \langle\vec u+\vec y,\vec v\rangle-\langle\vec u,\vec v\rangle=\langle\vec y,\vec v\rangle,
		$$
		and thus $ f^L(\vec y) = \langle\vec y,\vec v\rangle$ which directly leads to $\vec v\in\partial f^L(\vec y)$. The proof is completed. 
	\end{proof}

	%
	
	It has been revealed in \cite{chang2021lovasz} that the convex one-homogeneous functions: $\|\vec x\|_{\infty}$, $\|\vec x\|_{1,d}$, $I^+(\vec x)$ and $N(\vec x)$ in Eq.~\eqref{conti-theta:balance} are the set-pair Lov{\'a}sz extensions of set-pair functions: $1$, $\sum\limits_{i\in V_1\cup V_2}d_i$, $2|E(V_1, V_1)|+2|E(V_2, V_2)|+|E(V_1\cup V_2, (V_1\cup V_2)^c)|$ and $\sum\limits_{X\in\{V_1, V_2\}}\min\limits_{Y\in\{X,X^c\}}\vol(Y) $ for $(V_1, V_2)\in\mathcal{P}_2(V)$, respectively.  From Lemma~\ref{lem:subgradient}, for any $\vec x\in\mathbb{R}^n$ and its induced ternary vector $\vec y=\vec 1_{V_0^+(\vec x),V_0^-(\vec x)}$, it immediately arrives at
	\begin{align} 
		\partial \|\vec x\|_{\infty} &\subseteq \partial \|\vec y\|_{\infty}, \quad \partial \|\vec x\|_{1,d} \subseteq \partial \|\vec y\|_{1,d},\label{eq:sub_ternary1} \\
		\partial I^+(\vec x) &\subseteq \partial I^+(\vec y), \quad \partial N(\vec x) \subseteq \partial N(\vec y). \label{eq:sub_ternary2}
	\end{align}
	
	We are now ready for proving Eq.~\eqref{conti-theta:balance} for $h_{\theta}(G)$ and Eq.~\eqref{conti2-prob:balance} can be readily obtained by setting $\theta=1$. For convenience, we denote the objective function of Eq.~\eqref{conti-theta:balance} by $T(\vec x)$, and let 
	\begin{equation}
		\label{eq:l_theta}
		L_{\theta}(\vec x,\vec s)=\theta\|\vec x\|_{\infty} +\frac{1-\theta}{e}\|\vec x\|_{1,d}-\langle\vec x,\vec s\rangle
	\end{equation}
	be a linear function for $\vec x\in\mathbb{R}^n$ and $\vec s\in\mathbb{R}^n$. It can be easily found that, when $\theta=1$, $L_1(\vec x,\vec s^k)$
	recovers the objective function of \eqref{eq:twostep_x2_2} and hereafter we always neglect the subscript $\theta$ in this situation and use $L(\vec x,\vec s^k)$ for simplicity.

	\begin{proof} ({\bf Proof of Eq.~\eqref{conti-theta:balance}})	 
		Supposing $(V_1,V_2)\in\tc(V)$ being the optimal solution to $h_{\theta}(G)$, the ternary vector $\vec 1_{V_1,V_2}$ satisfies $T(\vec 1_{V_1,V_2})=h_{\theta}(G)$, which derives $h_{\theta}(G) \geq\min\limits_{\text{nonconstant~}\vec x\in\mathbb{R}^n} T(\vec x)$. Thus it remains to prove the contrary. 
		
		Assuming $\vec x^*\in\argmin\limits_{\text{nonconstant~}\vec x\in\mathbb{R}^n} T(\vec x)$ and $ r^*=T(\vec x^*)$,
		for any fixed $\vec s\in\partial H_{r^*}(\vec x^*)$ (see Eq.~\eqref{eq:Hr}), $H_{r^*}(\vec x)$ is a convex and first degree homogeneous function, indicating $\langle\vec x^*,\vec s\rangle= H_{r^*}(\vec x^*)$ and $\langle\vec x,\vec s\rangle\leq H_{r^*}(\vec x)$ for $\forall\, \vec x\in\mathbb{R}^n$. Therefore, we have $L_{\theta}(\vec x^*,\vec s)=0$ and
		\begin{equation}
			\label{eq:l_theta_geq0}
			L_{\theta}(\vec x,\vec s)\geq \theta\|\vec x\|_{\infty} +\frac{1-\theta}{e}\|\vec x\|_{1,d}-H_{r^*}(\vec x)\geq 0 = L_{\theta}(\vec x^*,\vec s),\,\,\forall\,\vec x\in\mathbb{R}^n,
		\end{equation}
		thereby implying that 
		\begin{equation}
			\label{eq:min_s}
			\vec x^*\in\argmin\limits_{\vec x\in\mathbb{R}^n} L_{\theta}(\vec x,\vec s)\Leftrightarrow	\vec 0\in \theta\partial\|\vec x^*\|_{\infty}+\frac{1-\theta}{e}\partial \|\vec x^*\|_{1,d}-\vec s.
		\end{equation}
		Considering the ternary vector $\vec y=\vec 1_{V_0^+(\vec x^*),V_0^-(\vec x^*)}$,
		it is readily verified through Eqs.~\eqref{eq:sub_ternary1} and \eqref{eq:sub_ternary2} that $\vec y$ also satisfies Eq.~\eqref{eq:min_s}. Therefore, we have $r^*=T(\vec y)\geq h_{\theta}(G)$ which ends the proof. 
	\end{proof}

	It should be pointed out that Eq.~\eqref{conti-theta:balance} can also be directly proved by the set-pair Lov{\'a}sz extension \cite{chang2021lovasz}. 
	Let $\Omega_1=\{\vec x\in\mathbb{R}^n:\max_i x_i+\min_i x_i=0\text{ and }\|\vec x\|_1=1\}$ and $\Omega_2=\{\vec x\in\mathbb{R}^n:\min_i |x_i|=0\text{ and }\|\vec x\|_1=1\}$, both of which are compact. Accordingly, we are able to reduce the non-compact constraint ``nonconstant $\vec x\in\mathbb{R}^n$" to $\vec x\in\Omega_1\cup\Omega_2$, but preserve all feasible solutions originally contained in $\tc(V)$ defined in Eq.~\eqref{ternary-cut},
	by fully exploiting the fact that $\|\vec x\|_{\infty}$, $\|\vec x\|_{1,d}$, $I^+(\vec x)$ and $N(\vec x)$ in Eq.~\eqref{conti-theta:balance} are all first degree homogeneous. After that, we may solve Eq.~\eqref{conti-theta:balance} via the following two-step Dinkelbach iterative scheme \cite{dinkelbach1967nonlinear}
	\begin{subequations}
		\label{iter0}
		\begin{numcases}
			\vec x^{k+1}={\argmin\limits_{\vec x\in\Omega_1\cup\Omega_2} \{\theta\|\vec x\|_{\infty}+\frac{1-\theta}{e}\|\vec x\|_{1,d}-H_{r^k}(\vec x)\}},\label{iter0-1}\\
			r^{k+1}=T(\vec x^{k+1}). \label{iter0-2}
		\end{numcases}
	\end{subequations}

	%

	\begin{theorem}[global convergence]
		\label{thm:conver_1}
		The sequence $\{r^k\}$ generated by the two-step iterative scheme \eqref{iter0} from any initial point $\vec x^0\in \mathbb{R}^n\setminus \{\vec0\}$ decreases monotonically to $h_{\theta}(G)$.
	\end{theorem}
	
	\begin{proof}
		Let $r_{\min}=\min_{\vec x\in\Omega_1\cup\Omega_2} T(\vec x)$, indicating $r_{\min}=h_{\theta}(G)$.
		The definition of $\vec x^{k+1}$ in Eq.~\eqref{iter0-1} implies
		\begin{equation*}
			\begin{aligned}
				0&=\theta\|\vec x^k\|_{\infty}+\frac{1-\theta}{e}\|\vec x^k\|_{1,d}-H_{r^k}(\vec x^k)\\
				&\ge \theta\|\vec x^{k+1}\|_{\infty}+\frac{1-\theta}{e}\|\vec x^{k+1}\|_{1,d}-H_{r^k}(\vec x^{k+1}),\\
			\end{aligned}
		\end{equation*}
		which means
		\begin{equation*}
			r^k\ge r^{k+1}\ge r_{\min}, \,\,\, \forall\, k\in \mathbb{N}^+.
		\end{equation*}
		Thus we have
		\[
		\exists\, r^* \in [r_{\min},r^0]\,\, \text { s.t. }\, \lim\limits_{k\to+\infty}r^k=r^*,
		\]
		and it suffices to show $r_{\min}\ge r^*$. To this end, we denote
		\begin{equation*}
			f(r)=\min\limits_{\vec x\in\Omega_1\cup\Omega_2}(\theta\|\vec x\|_{\infty}+\frac{1-\theta}{e}\|\vec x\|_{1,d}-H_{r}(\vec x)),
		\end{equation*}
		which must be continuous on $\mathbb{R}$ by the compactness of $\Omega_1\cup\Omega_2$. Note that
		\begin{align*}
			f(r^k)&=\theta\|\vec x^{k+1}\|_{\infty}+\frac{1-\theta}{e}\|\vec x^{k+1}\|_{1,d}-H_{r^k}(\vec x^{k+1})\\
			&=\frac{N(\vec x^{k+1})}{e}\cdot(r^{k+1}-r^{k})\to 0 \quad \text{as} \quad k\to +\infty,
		\end{align*}
		then we have
		\begin{equation*}
			f(r^*)=\lim\limits_{k\to+\infty}f(r^k)=0,
		\end{equation*}
		and thus
		\begin{equation*}
			\theta\|\vec x\|_{\infty}+\frac{1-\theta}{e}\|\vec x\|_{1,d}-H_{r^*}(\vec x)\ge 0, \,\,\forall\,\vec x\in\Omega_1\cup\Omega_2,
		\end{equation*}
		which directly yields: $\forall\, \vec x\in \Omega_1\cup\Omega_2$, $T(\vec x) \geq r^*$. Namely,  $r_{\min}\ge r^*$. We complete the proof. 
	\end{proof}

	It can be easily verified that setting $\theta=1$ in Eq.~\eqref{iter0} gives 
	\begin{subequations}
		\label{iter00}
		\begin{numcases}
			\vec x^{k+1}={\argmin\limits_{\vec x\in\Omega_1\cup\Omega_2} \{\|\vec x\|_{\infty}-H_{r^k}(\vec x)\}},\label{iter00-1}\\
			r^{k+1}=B(\vec x^{k+1}), \label{iter00-2}
		\end{numcases}
	\end{subequations}
	and thus Theorem~\ref{thm:conver_1} also implies the above iteration \eqref{iter00} solves Eq.~\eqref{conti2-prob:balance} for $h(G)$.

	\section{Simple inverse power method}
	\label{sec:algorithm}
	

	The global convergence established in Theorem~\ref{thm:conver_1} means that the NP-hard problem $h_{\theta}(G)$ is equivalent to the two-step Dinkelbach iterative scheme \eqref{iter0}. However, its subproblem \eqref{iter0-1} is still non-convex and can not be solved in polynomial time.
	In this work, we apply a subgradient relaxation
	\begin{equation}\label{eq:rex}
		H_r (\vec x)\geq  H_r(\vec y) +\langle\vec x-\vec y, \vec s\rangle = \langle\vec x, \vec s\rangle, \, \forall\, \vec s\in\partial H_r (\vec y), \, \forall\, \vec x, \vec y \in \mathbb{R}^n, \, r\ge 0, 
	\end{equation}
	into it, and obtain $\mathbf{SIP_{\theta}}$ --- a three-step inverse power method:
	\begin{subequations}
		\label{iter2-cheeger_2}
		\begin{numcases}{}
			\vec x^{k+1}=\mathop{\argmin}\limits_{\|\vec x\|_1\leq 1} \{L_{\theta}(\vec x,\vec s^k)\},
			\label{eq2:twostep_x2_2}
			\\
			r^{k+1}=T(\vec x^{k+1}),
			\label{eq2:twostep_r2_2}
			\\
			\vec s^{k+1}\in\partial H_{r^{k+1}}(\vec x^{k+1}),
			\label{eq2:twostep_s2_2}
		\end{numcases}
	\end{subequations}
	and the condition for exiting the iteration is that the objective function $T(\vec x)$ is unchanged (see Fig.~\ref{flowchart}). In parallel, $\mathbf{SIP}$ in Eq.~\eqref{iter1-cheeger_2} is relaxed from Eq.~\eqref{iter00} and can be also reduced from $\mathbf{SIP_{\theta}}$ in \eqref{iter2-cheeger_2} with $\theta=1$. Compared with the Dinkelbach iteration \eqref{iter0}, $\mathbf{SIP}_{\theta}$ holds advantages at least in two aspects after setting aside the relaxation temporarily.   One is that the subproblem \eqref{eq2:twostep_x2_2} can be solved analytically in a simple manner (see Section~\ref{sec:exact_prob}); the other is that the iterative values of the objective function still remain monotonically decreasing thanks to
	\begin{equation}
		\label{eq:decrease}
		\begin{aligned}
			r^{k+1}-r^k&=\frac{\theta e\|\vec x^{k+1}\|_{\infty}+(1-\theta)\|\vec x^{k+1}\|_{1,d}-eH_{r^k}(\vec x^{k+1})}{N(\vec x^{k+1})}\\
			&\leq\frac{eL_{\theta}(\vec x^{k+1},\vec s^k)}{N(\vec x^{k+1})}\leq\frac{eL_{\theta}(\vec x^k,\vec s^k)}{N(\vec x^{k+1})}=0, \,\,\forall\,\vec s^k\in\partial H_{r^k}(\vec x^k).
		\end{aligned}
	\end{equation}
	We make the best use of both features to establish an elaborate subgradient selection in Section~\ref{sec:subgradient}, and ensure that $\mathbf{SIP}_{\theta}$ and $\mathbf{SIP}$ converge to local optima in finite iterations (see Section~\ref{sec:local}), though the global convergence is degraded by the relaxation \eqref{eq:rex}.

	\subsection{\textbf{Inner subproblem solution}}
	\label{sec:exact_prob}
	
	
	Let 
	\begin{align}
		\vec l^k &= \frac{\theta-1}{e}\vec d+|\vec s^k|, \label{eq:vec_l} \\
		\Omega_{\vec s^k} &= \left\{i\in V: l_i^k\geq 0\right\}, \label{eq:gamma_region}
	\end{align}
	where the absolute value of a vector is taken in an element-wise manner. Then 
	the subproblem~\eqref{eq2:twostep_x2_2} can be rewritten into
	\begin{equation}
		\label{eq:theta-cheeger0}
		\vec x^{k+1}=\mathop{\argmin}\limits_{\|\vec x\|_1\leq 1} \left\{\theta\|\vec x\|_{\infty}-\langle|\vec x|, \vec l^k\rangle\right\},
	\end{equation}
	and it can be readily verified that $x_i^{k+1}=0$ for $i\in V\backslash\Omega_{\vec s^k}$.
	That is, we only need to find solutions of Eq.~\eqref{eq:theta-cheeger0} in $\Omega_{\vec s^k}$ below, 
	which can be achieved by the method proposed in \cite{SZZ2018}. Before that, we need the following convention on the sign function and its set-valued extension:
	\begin{equation}
		\sign(t) = \left\{
		\begin{aligned}
			&1,&t\geq 0,\\
			&-1,&t<0,
		\end{aligned}
		\right.,\quad \sgn(t) = \left\{
		\begin{aligned}
			&\{1\},&t>0,\\
			&\{-1\},&t<0, \\
			&[-1, 1],&t=0,
		\end{aligned}
		\right.
	\end{equation}
	and for $\vec x=(x_1,x_2,\ldots,x_n)\in\mathbb{R}^n$, 
	we are able to define corresponding vectorized versions
	in an element-wise manner: 
	\begin{align}
		\sign(\vec x)&=(s_1, s_2,\ldots,s_n),\,\,\,s_i = \sign(x_i),\,i = 1,2,\ldots,n,\\
		\sgn(\vec x)&=\{(s_1, s_2,\ldots,s_n) \big| s_i\in \sgn(x_i),\,i = 1,2,\ldots,n\}.
	\end{align}

	\begin{lemma}
		\label{prop:rs}
		The simple iterative scheme~\eqref{iter2-cheeger_2} with $\theta\in(0,1]$ satisfies $\| \vec l^k \|_1\geq \theta$. In particular, 
		(1) $\| \vec l^k \|_1 =  \theta \Leftrightarrow 
		\frac{|\vec x^k|}{\|\vec x^k\|_{\infty}}\in \sgn(\vec l^k)$ and $\forall\,i\in V, x_i^ks_i^k\geq 0$; (2) $\| \vec l^k \|_1 >  \theta \Leftrightarrow L_\theta(\vec x^{k+1}, \vec s^k) = \theta\|\vec x^{k+1}\|_{\infty}-\langle\vec l^k, |\vec x^{k+1}|\rangle < 0$. In addtion, if $\| \vec l^k \|_1 >  \theta$, then $r^{k} > r^{k+1}$. 
	\end{lemma}
	\begin{proof}
		This is a direct consequence of Lemma 3.1 in \cite{SZZ2018}. 
	\end{proof}
	
	
	\begin{lemma}\label{Thm:exact_solution}
		The solution of Eq.~\eqref{eq:theta-cheeger0} with $\theta\in(0,1]$ can be obtained in the following two steps. 
		\begin{itemize}
			\item \textbf{Step 1} Sort all the components of $|\vec l^k|$ in a descending order. Without loss of generality,  we assume  
			$|l_1^k|\geq |l_2^k| \geq\dots\geq |l_n^k| \geq |l_{n+1}^k|=0$, and calculate $m_0=\min\{m: A_m^k>\theta\}$
			and $m_1=\max\{m: A_{m-1}^k<\theta\}$ with $A_m^k=\sum_{j=1}^m(|l_j^k|-|l_{m+1}^k|)$ for $m=1, 2, \dots, n$. 
			Here we have introduced an auxiliary element $l_{n+1}^k=0$ for convenience. 
			\item \textbf{Step 2} The solution construction is divided into two cases:
			\begin{itemize}
				\item If $\| \vec l^k \|_1 >\theta$, we have $x_i^{k+1}= \sign(s_i^k) z_i / \|\vec z\|_1$ for $i=1, 2, \dots, n$,  
				where the vector $\vec z$ reads
				\begin{equation}
					\label{eq:exact_z}
					z_i=\left\{
					\begin{aligned}
						&1, &1\leq i\leq m_1,\\
						&0, &m_0\leq i\leq n,\\
						&[0,1], &m_1<i<m_0;
					\end{aligned}
					\right.
				\end{equation}
				\item If $\| \vec l^k \|_1 =\theta$, the solution satisfies 
				$\frac{|\vec x^{k+1}|}{\|\vec x^{k+1}\|_{\infty}}\in \sgn(\vec l^k)$,  $\|\vec x^{k+1}\|_1=1$,
				and $\forall\,i\in V, x_i^{k+1}s_i^{k}\geq 0$.
			\end{itemize}
		\end{itemize}
		Moreover, the minimizer $\vec x^{k+1}$ satisfies 
		\begin{equation}\label{eq:m}
			\frac{\langle |\vec x^{k+1}|,|\vec l^k|\rangle}{\|\vec x^{k+1}\|_{\infty}} \ge \sum_{i=1}^{m} |l_i^k|,
		\end{equation}
		where $m =m_0$ for $\| \vec l^k \|_1>\theta$, and $m=n$ for $\| \vec l^k \|_1=\theta$. 
	\end{lemma} 
	\begin{proof}
		The solution construction can be readily achieved using Theorem 3.1 and Lemma 3.2 in \cite{SZZ2018} and Eq.~\eqref{eq:m} comes from Corollary 3.1 in \cite{SZZ2018}. Here we skip the details for saving space.
	\end{proof}
	
	Lemma~\ref{Thm:exact_solution} demonstrates that Eq.~\eqref{eq2:twostep_x2_2} can be analytically solved via a simple sort operation, thereby reducing the computational cost of $\mathbf{SIP}_{\theta}$, as detailed in Section~\ref{SIP-IP}. In practice, $\mathbf{SIP}_{\theta}$ prefers to choose $\vec x^{k+1}$ on the boundary of the closed solution pool of Eq.~\eqref{eq:theta-cheeger0} after that the $k$-th iteration is completed with 	$\vec s^k\in\partial H_{r^k}(\vec x^k)$, leading to 
	\begin{equation}
		\label{eq:teranry_vec}
		\frac{\vec x^{k+1}}{\|\vec x^{k+1}\|_{\infty}}\in\{-1,0,1\}^n
	\end{equation}
	from Lemma~\ref{Thm:exact_solution}. Specifically, \textbf{Step 2} in Lemma~\ref{Thm:exact_solution} suggests to select $z_i=X_{0,1}$  for $i\in (m_1, m_0)$ if $\|\vec l^k\|_1>\theta$ (see Eq.~\eqref{eq:exact_z}), where $X_{a,b}\in\{a,b\}$ denotes a random variable with an equal probability of being either $a$ or $b$. As for $\|\vec l^k\|_1=\theta$, Lemma~\ref{Thm:exact_solution} results in 
	\begin{equation}
		\label{eq:exact_x}
		\frac{x_i^{k+1}}{\|\vec x^{k+1}\|_{\infty}}\in\left\{
		\begin{aligned}
			&\{\sign(s_i^k)\}, &\text{if }l_i^k>0,\\
			&[a_i,b_i], &\text{if }l_i^k=0,
		\end{aligned}
		\right.
		\quad
		[a_i,b_i]=\left\{
		\begin{aligned}
			&[-1,1],&\text{if }s_i^k=0,\\
			&[0,1], &\text{if }s_i^k>0,\\
			&[-1,0],&\text{if }s_i^k<0.
		\end{aligned}
		\right.
	\end{equation}
	Accordingly, we choose $x_i^{k+1}=\sign(s_i^k)$ for $i\in \Gamma_1^k$ and it remains to determine $x_i^{k+1}$ for $i \in \Gamma_0^k$ such that the nonconstant constraint is satisfied, when $\Gamma_1^k=\{i\in V:\,l_i^k>0\}$ and $\Gamma_0^k=\{i\in V:\,l_i^k=0\}$. To this end, the selection can be specified in the following two cases.
	\begin{itemize}
		\item\textbf{Case 1} If $|\Gamma_1^k|\geq 1$, first randomly select $j\in \Gamma_1^k$ and $i\in\Gamma_0^k$, and we ensure $x_i^{k+1}\neq x_j^{k+1}$, i.e.
		$$x_i^{k+1}=\left\{
		\begin{aligned}
			&-1, &\text{if }s_i^k=0\text{ and }\sign(s_j^k)=1,\\
			&0, &\text{if }s_i^k>0\text{ and }\sign(s_j^k)=1,\\
			&X_{-1,0}, &\text{if }s_i^k<0\text{ and }\sign(s_j^k)=1,\\
			&1, &\text{if }s_i^k=0\text{ and }\sign(s_j^k)=-1,\\
			&X_{0,1}, &\text{if }s_i^k>0\text{ and }\sign(s_j^k)=-1,\\
			&0, &\text{if }s_i^k<0\text{ and }\sign(s_j^k)=-1.
		\end{aligned}
		\right.
		$$
		Later, select $x_t^{k+1}=X_{a_t,b_t}$ for  $t\in\Gamma_0^k\backslash
		\{i\}$ and then scale $\vec x^{k+1}= \frac{\vec x^{k+1}}{\|\vec x^{k+1}\|_1}$.
		\item\textbf{Case 2} If $|\Gamma_1^k|=0$, first randomly select $i\in\Gamma_0^k$ and choose $x_i^{k+1}=X_{a_i,b_i}$, then randomly select $j\in\Gamma_0^k\backslash
		\{i\}$ and ensure $x_i^{k+1}\neq x_j^{k+1}$, i.e.
		$$x_j^{k+1}=\left\{
		\begin{aligned}
			&-1, &\text{if }s_j^k=0\text{ and }X_{a_i,b_i}=1,\\
			&0, &\text{if }s_j^k>0\text{ and }X_{a_i,b_i}=1,\\
			&X_{-1,0}, &\text{if }s_j^k<0\text{ and }X_{a_i,b_i}=1,\\
			&1, &\text{if }s_j^k=0\text{ and }X_{a_i,b_i}=-1,\\
			&X_{0,1}, &\text{if }s_j^k>0\text{ and }X_{a_i,b_i}=-1,\\
			&0, &\text{if }s_j^k<0\text{ and }X_{a_i,b_i}=-1,\\
			&X_{-1,1}, &\text{if }s_j^k=0\text{ and }X_{a_i,b_i}=0,\\
			&1, &\text{if }s_j^k>0\text{ and }X_{a_i,b_i}=0,\\
			&-1, &\text{if }s_j^k<0\text{ and }X_{a_i,b_i}=0.
		\end{aligned}
		\right.
		$$
		Later, select $x_t^{k+1}=X_{a_t,b_t}$ for $t\in\Gamma_0^k\backslash
		\{i,j\}$ and scale $\vec x^{k+1}= \frac{\vec x^{k+1}}{\|\vec x^{k+1}\|_1}$.
	\end{itemize}
		Therefore, $\mathbf{SIP}_{\theta}$ must produce nonconstant ternary valued solutions contained in $\tc(V)$ during the iterations which are also in consistence with the discrete nature of $h_{\theta}(G)$. In fact,
	when $\mathbf{SIP}_{\theta}$ fails to reduce $T(\vec x)$, it is still required to output a solution within $\tc(V)$ rather than other solutions outside $\tc(V)$,  for instance, the constant ones. Particularly, when $\theta=1$, it is easily verified from Eq.~\eqref{eq:exact_x} that $\mathbf{SIP}$ yields nonconstant binary valued solutions, thereby maintaining alignment with the discrete nature of $h(G)$. 
		
	Finally, we would like to point out that the above analytical solution to the inner subproblem~\eqref{eq2:twostep_x2_2}
	is of great importance in ensuring the local optimality of $\mathbf{SIP}_{\theta}$ (see Theorem~\ref{thm:conver_3theta}).
	In particular, when $\theta=1$, the output of $\mathbf{SIP}$ serves as a discrete local optimum (see Theorem~\ref{thm:discrete}) whereas that of $\mathbf{IP}$ may not. 
	By contrast,  the inner subproblem~\eqref{eq:twostep_x2} in $\mathbf{IP}$ necessitates a convex solver with nonzero precision, making it a challenge to explore its locality, let alone maintain a discrete local optimum (see Example~\ref{ex:ip}).

	\subsection{\textbf{Subgradient selection}}
	\label{sec:subgradient}
	
	
	We demonstrate the subgradient selection for $\mathbf{SIP}$ in Eq.~\eqref{iter1-cheeger_2} at the first place. 
	According to Eqs.~\eqref{eq:rex} and \eqref{eq:m}, the monotonically decreasing property of $r^k$ suggests a direct relationship with arbitrary subgradient $\vec s^k\in\partial H_{r^k}(\vec x^k)$: 
	\begin{equation}\label{eq:rsr}
		\begin{aligned}
			r^k-r^{k+1}&=\frac{e\|\vec x^{k+1}\|_{\infty}}{N(\vec x^{k+1})}\left(\frac{H_{r^k}(\vec x^{k+1})}{\|\vec x^{k+1}\|_{\infty}}-1\right)\\
			&\geq \frac{e\|\vec x^{k+1}\|_{\infty}}{N(\vec x^{k+1})}\left(\frac{\langle\vec x^{k+1},\vec s^k\rangle}{\|\vec x^{k+1}\|_{\infty}}-1\right)\\
			&\geq \frac{e\|\vec x^{k+1}\|_{\infty}}{N(\vec x^{k+1})}\left(\sum_{i\in\iota} |s_i^k|-1\right)\geq 0,
		\end{aligned}
	\end{equation} 
	which involves an index set $\iota\subset\{1,2,\ldots,n\}$ with a size of $m$ specified in Eq.~\eqref{eq:m}. Therefore, $r^{k+1}<r^k$ holds if $\vec s^k$ satisfies that there exists $i\in V$ such that $x_i^k/\|\vec x^k\|_{\infty}\notin\sgn(s_i^k)$, which derives from 
	\begin{equation}
		\label{neq:strict_r}
		\exists\, i\in V,\,\frac{x_i^k}{\|\vec x^k\|_{\infty}}\notin\sgn(s_i^k)\Rightarrow \frac{\vec x^{k}}{\|\vec x^k\|_{\infty}}\notin\sgn(\vec s^k)\Rightarrow\|\vec s^k\|_1>1\Rightarrow r^{k+1}<r^k
	\end{equation}
	as stated in Lemma~\ref{prop:rs} and Eq.~\eqref{eq:rsr}. In this case, it motivates us to always select $\vec s^k\in \partial H_{r^k}(\vec x^k)$ that ensures Eq.~\eqref{neq:strict_r}, provided that such subgradient exists. 
	Moreover, maximizing $\|\vec s^k\|_1$ as much as possible within the convex region $\partial H_{r^k}(\vec x^k)$ may contribute to a substantial decrease at each iteration. Thus, we will select such subgradient $\vec s^k$ on the boundary of $\partial H_{r^k}(\vec x^k)$, the characterization of which depends on the characterization of $\partial N(\vec x)$ and $\partial I^+(\vec x)$ by Eq.~\eqref{eq:Hr}. In this regard, we call such strategy a \emph{boundary-detected subgradient selection}.
	\\

	$\bullet$ Characterization of $\partial N(\vec x)$
	\\

	Let
	\begin{align}
		\median(\vec x)&=\argmin_{c\in\mathbb{R}}\sum_{i\in V}\mu_i|x_i-c|,\label{eq:median}\\
		S^{\pm}(\vec x)&=\{i\in V \big| x_i=\pm \|\vec x\|_{\infty}\}, \label{eq:x_category1}\\
		S^<(\vec x)&=\{i\in V\big | |x_i|< \|\vec x\|_{\infty}\},\label{eq:x_category2}\\
		S^\alpha(\vec x)&=\{i\in V \big |x_i=\alpha\},\quad\alpha\in \median(\vec x).\label{eq:x_category_alpha}
	\end{align}
	Then the $i$-th component,  $(\partial N(\vec x))_i$,  satisfies 
	\begin{equation}
		(\partial N(\vec x))_i = [a_{i}^L, a_{i}^R], 
	\end{equation}
	where 
	\begin{align}
		a_{i}^L&=
		\left\{\begin{array}{ll}
			{\max\{A-B+\mu_i,-\mu_i\},} & {\text { if } i\in S^{\alpha}(\vec x)\text{ and }|	S^\alpha(\vec x)|\geq 2,} \\ 
			{A,} & {\text { if } i\in S^{\alpha}(\vec x)\text{ and }|	S^\alpha(\vec x)|\leq 1,}\\
			{\mu_i\sign(x_i-\alpha),} & {\text { if }i\in V\setminus S^{\alpha}(\vec x),}
		\end{array}
		\right. \label{eq:a_bound_l}\\
		a_{i}^R &= 
		\left\{\begin{array}{ll}
			{\min\{A+B-\mu_i,\mu_i\},} & {\text { if } i\in S^{\alpha}(\vec x)\text{ and }|	S^\alpha(\vec x)|\geq 2,} \\ 
			{A,} & {\text { if } i\in S^{\alpha}(\vec x)\text{ and }|	S^\alpha(\vec x)|\leq 1,}\\
			{\mu_i\sign(x_i-\alpha),} & {\text { if }i\in V\setminus S^{\alpha}(\vec x),}
		\end{array}
		\right. \label{eq:a_bound_r}	 \\
		A &=\sum_{x_i<\alpha}\mu_i-\sum_{x_i>\alpha}\mu_i,\quad
		B = \sum_{x_i=\alpha}\mu_i. \label{eq:v_ab}
	\end{align}
	\\

	$\bullet$ Characterization of $\partial I^+(\vec x)$
	\\
	
	A direct algebraic calculation gives 
	\begin{equation}
		(\partial I^+(\vec x))_i = [p_i-q_i,p_i+q_i],
	\end{equation}
	where 
	\begin{align}
		p_i &= \sum_{j:\{i,j\}\in E} w_{ij}\sign(x_i+x_j)-q_i, \\
		q_i &= \sum_{j\in\nen(i,\vec x)} w_{ij}, \\
		\nen(i,\vec x) &= \{j:\{i,j\}\in E\text{ and }x_i+x_j=0\}.
	\end{align}
	Here $\nen(i,\vec x)$ denotes the set of ``negative equal neighbors" of the $i$-th vertex $i$ on $\vec x$. 	\\

	$\bullet$ Characterization of $\partial H_r(\vec x)$
	\\
	
	According to Eq.~\eqref{eq:Hr}, using the property of convex functions leads to

	%
	\begin{equation}
		\label{eq:partial_h}
		\partial H_r(\vec x)=\frac{1}{e}\left(\partial I^+(\vec x)+r\partial N(\vec x)\right),
	\end{equation} 
	implying that each $\vec s\in\partial H_r(\vec x)$ can be constructed as   
	\begin{equation}
		\label{eq:sub_s}
		\vec s=\frac{1}{e}(\vec u+r\vec v), \quad \vec u\in\partial I^+(\vec x), \quad \vec v\in\partial N(\vec x).
	\end{equation}
	As a result, the $i$-th component $(\partial H_r(\vec x))_i$ also appears in the form of an interval. \\

	In order to search for the subgradient $\vec s$ that ensures Eq.~\eqref{neq:strict_r} (the superscript $k$ in $\vec s^k$, $\vec l^k$, $\vec x^k$ and $r^k$ is neglected for simplicity for the rest of Section~\ref{sec:subgradient}) on the boundary of $\partial H_r(\vec x)$, 
	we introduce a boundary indicator $\vec b=(b_1,\ldots,b_n)\in\mathbb{R}^n$
	\begin{align}
		b_{i}&=\left\{\begin{array}{ll}{p_{i} \mp q_{i}+r a_i,} & {\text { if } i \in S^{ \pm}(\vec x),} \\ {p_{i}+r a_i+\operatorname{sign}\left(p_{i}+r a_i\right) q_{i},} & \text { if } i \in S^{<}(\vec x),
		\end{array}\right.\label{u1}
	\end{align}
	with
	\begin{align}
		a_i&=
		\begin{cases}
			\mu_i\sign(x_i-\alpha),&\text{if $i\in V\setminus S^{\alpha}(\vec x)$},\\
			a^L_i,&\text{if $i\in S^{\alpha}(\vec x)\cap S^+(\vec x)$ and $|S^\alpha(\vec x)|\geq 2$},\\
			a^R_i,&\text{if $i\in S^{\alpha}(\vec x)\cap S^-(\vec x)$ and $|S^\alpha(\vec x)|\geq 2$},\\  \argmax\limits_{t\in\{a^L_i,a^R_i\}}\{|p_i+rt|\} ,&\text{if $i\in S^{\alpha}(\vec x)\cap S^<(\vec x)$ and $|S^\alpha(\vec x)|\geq 2$},\\
			A,&\text{if $i\in S^{\alpha}(\vec x)$ and $|S^\alpha(\vec x)|\leq 1$},
		\end{cases}
		\label{hh}	
	\end{align}
	and each $b_i$ sits on the boundary of $(\partial H_r(\vec x))_i$.  Thus    
	for any $i\in V$, it can be easily verified that 
	\begin{equation}
		\label{eq:desire_v}
		\exists\, i\in V,\,\frac{x_i}{\|\vec x\|_{\infty}}\notin\sgn(s_i) \Rightarrow \frac{x_i}{\|\vec x\|_{\infty}}\notin\sgn(b_i).
	\end{equation}
	Consequently, we denote the vertex $i$ as the ``desired" one hereafter if it satisfies Eq.~\eqref{eq:desire_v}. 
	For any $i\in V$, let
	\begin{equation}
		\chi(i)=\left\{\begin{array}{ll}
			{\mp 1,} & {\text { if } x_i\in S^{\pm}(\vec x),} \\ 
			{\delta_{a_i}(a_i^R)-\delta_{a_i}(a_i^L),} & {\text { if } x_i\in S^{<}(\vec x),}
		\end{array}
		\right.\\
	\end{equation}
	be the preferred sign for $s_i\in(\partial H_r(\vec x))_i$, leading to $b_i\chi(i)>0$  if $i$ is a desired vertex,
	where the indicator function $\delta_{t}(x)$ returns $1$ only if $x=t$ and otherwise vanishes. We would like to use 
	\begin{equation}
		\label{eq:wanted_ub}
		V_{\vec b}=\argmax\limits_{i\in S^+(\vec x)}\{|b_i|:\,b_i<0\}\cup \argmax\limits_{i\in S^-(\vec x)}\{|b_i|:\,b_i>0\}\cup \argmax\limits_{i\in S^<(\vec x)}\{|b_i|:\,b_i\neq 0\}
	\end{equation}
	to collect the desired vertices each of which corresponds to a component of $\vec b$ with largest absolute value in  $S^+(\vec x)$, $S^-(\vec x)$, and $S^<(\vec x)$, respectively.
	Accordingly, the essence of subgradient selection for $\mathbf{SIP}$ is to dig out at least one desired vertex $i$, if there exists any, and then select $s_i=b_i/e$. The process can now be reformulated with more details into the following three steps. 
	
	\begin{itemize}
		\item \textbf{Step 1:} If $V_{\vec b} = \emptyset$ skip this step and go to \textbf{Step 2} directly; Otherwise randomly select $i^*\in V_{\vec b}$.  
		Then,  calculate
		\begin{equation}
			\label{star_u_theta}
			u_{i^*}=p_{i^*}+\sum\limits_{j\in\nen(i^*,\vec x)}w_{i^*j}\chi(i^*),
		\end{equation}
	and for each neighbor $i\in\nen(i^*,\vec x)$, given $\sigma\in\Sigma_{\vec b}(\vec x)$,  obtain
	\begin{equation}\label{u_neighbor}
		u_i=p_i+w_{ii^*}\chi(i^*)+\sum\limits_{j\in\nen(i,\vec x)\backslash\{i^*\}}w_{ij}\chi(\argmax\limits_{t\in\{i,j\}}\{\sigma^{-1}(t)\}),
	\end{equation}
	where $\Sigma_{\vec b}(\vec x)$ denotes a collection of  permutations of $\set{1,2,\ldots,n}$ such that for any $\sigma\in\Sigma_{\vec b}(\vec x)$, it holds $|b_{\sigma(1)}|\leq |b_{\sigma(2)}|\leq \cdots\leq |b_{\sigma(n)}|$.
	\item \textbf{Step 2:} For each remaining vertex $i$, the subgradient component follows 
	\begin{equation}\label{c1:u}
		u_i = p_i+\sum_{j\in\nen(i,\vec x)}w_{ij}\chi(\argmax_{t\in\{i,j\}}\{\sigma^{-1}(t)\}).
	\end{equation}
	\item \textbf{Step 3:} In the case of $|S^\alpha(\vec x)|\leq 1$, calculate $v_i=a_i$ according to Eq.~\eqref{hh} for each $i\in V$; 
	and in the case of $|S^\alpha(\vec x)|\geq 2$, select
	\begin{equation}
		\label{v2_1}
		v_i=
		\begin{cases}
			a_i,&\text{if }i\in \{j^*\} \cup V\backslash S^{\alpha}(\vec x),\\
			\frac{A-a_{j^*}}{B-\mu_{j^*}}\mu_i,&\text{if }i\in S^{\alpha}(\vec x)\backslash \{j^*\},
		\end{cases}
	\end{equation}
	with 
	$$
	j^*=\left\{\begin{array}{ll}
		{i^*,} & {\text { if } i^*\in S^{\alpha}(\vec x),} \\ 
		{\argmax\limits_{t\in S^{\alpha}(\vec x)}\{\sigma^{-1}(t)\},} & {\text { otherwise}.}
	\end{array}
	\right.
	$$
\end{itemize}
We would like to point out that $\sigma\in\Sigma_{\vec b}(\vec x)$ tends to satisfy $s_i\rightarrow b_i/e$ when $|b_i|$ is relatively large, potentially leading to a large value of $\|\vec s\|_1$ and a profound decrease in the iterative objective function values (see Eq.~\eqref{eq:rsr}).



\begin{remark}\label{re:separate}
	We will prove that the above boundary-detected subgradient selection ensures that  $\mathbf{SIP}$ converges locally in both continuous and discrete senses (see Theorems~\ref{thm:conver_3theta} and \ref{thm:discrete}). By contrast, a naive random selection can't. In fact,  our numerical results on G-set show that the numerical solutions with random subgradients are unlikely to be local minimizers with high probability and their solution quality is much worse than those obtained by $\mathbf{SIP}$ equipped with the  boundary-detected subgradient selection in all instances. 
\end{remark}

Next we consider the subgradient selection for $\mathbf{SIP}_{\theta}$ that serves as a perturbation of $\mathbf{SIP}$ as shown in Fig.~\ref{flowchart}. 
For any $i\in V$, it is easily verified that
\begin{align}
	\exists\, i\in \Omega_{\vec s},\,\frac{|x_i|}{\|\vec x\|_{\infty}}\notin\sgn(l_{i}) &\Rightarrow  \frac{|x_i|}{\|\vec x\|_{\infty}}\notin\sgn(b_{i}^\prime) \Rightarrow i\in S^<(\vec x),\label{eq:desire_v1}\\
	\exists\, i\in V,\, x_is_i<0 &\Rightarrow x_ib_i<0 \Rightarrow x_i\neq 0,\label{eq:desire_v2}
\end{align}
where 
\begin{equation}\label{eq:b'}
	\vec b^\prime=\frac{|\vec b|+(\theta-1)\vec d}{e},
\end{equation}
$\vec l$ and $\Omega_{\vec s}$ are defined in Eqs.~\eqref{eq:vec_l} and \eqref{eq:gamma_region}, respectively. 
According to Lemma~\ref{prop:rs}, the desired vertex $i$ turns to satisfy either Eq.~\eqref{eq:desire_v1} or Eq.~\eqref{eq:desire_v2}. Let
\begin{equation}
	\label{eq:wanted_ub1}
	\begin{aligned}
		V_{\vec b}^\prime=&\argmax\limits_{i\in S^+(\vec x)}\{|b_i|:\,b_i<0\}\cup \argmax\limits_{i\in S^-(\vec x)}\{|b_i|:\,b_i>0\}\cup \argmax\limits_{i\in S^<(\vec x)}\{|b_i^\prime|:\,b_i^\prime> 0\}\\
		&\cup\argmax\limits_{i\in S^<(\vec x)}\{|b_i|:\,|b_i^\prime|= 0\text{ and }x_ib_i<0\}.
	\end{aligned}
\end{equation}
Therefore,  the subgradient selection in $\mathbf{SIP}_{\theta}$ follows the same aforementioned three steps as in $\mathbf{SIP}$ except for replacing 
$V_{\vec b}$ with $V_{\vec b}^\prime$. Note that $V_{\vec b}^\prime$ reduces into $V_{\vec b}$ when $\theta=1$.

\subsection{\textbf{Local analysis}}
\label{sec:local}

The proposed  boundary-detected subgradient selection for $\mathbf{SIP}_{\theta}$ ensures 
that the objective function $T(\vec x)$ strictly decreases whenever possible as shown in Theorem~\ref{prop:S123_theta}.
Furthermore, we are able to prove that $\mathbf{SIP}_{\theta}$ converges to a local minimum $\vec x^*$ in the neighborhood 
\begin{equation}
	\label{neighbor}
	U(\vec x^*)=\left\{\vec y\in \mathbb{R}^n\,\, \Big|\,\, \|\frac{\vec y}{\|\vec y\|_{\infty}}-\frac{\vec x^*}{\|\vec x^*\|_{\infty}}\|_\infty<\frac12\right\}
\end{equation}
within finite iterations as demonstrated in  Theorem~\ref{thm:conver_3theta}. 
In addition, we will show that $\vec x^*$ also serves as a local minimum in $C_B$ for $\mathbf{SIP}$ (see Theorem~\ref{thm:discrete}) from a discrete perspective, where
\begin{align}
	C_B &= \{\vec x\in\mathbb{R}^n \big| B(\vec y)\geq B(\vec x), \,\, \forall\,\vec y\in\{R_i\vec x: i\in\{1,\ldots,n\}\}\},  \label{eq:CT}\\
	(R_i\vec x)_j& =
	\begin{cases}
		x_j, & j\ne i,\\
		-x_j, & j=i.
	\end{cases}\label{eq:Tialp}
\end{align}

\begin{theorem}[$\mathbf{SIP}_{\theta}$: strict descent guarantee]
\label{prop:S123_theta}
Suppose $\vec x^k$ and $r^k$ are produced by $\mathbf{SIP}_{\theta}$ at the $k$-th iteration and define the following function 
\begin{equation}
	\label{eq:sip_theta_thm}
	\widetilde{L}_{\theta}(\vec s)=\min_{\|\vec y\|_1\leq 1} L_{\theta}(\vec y, \vec s) 
\end{equation}
on the convex domain $\partial H_{r^k}(\vec x^k)$. 
Then we have
\[
\widetilde{L}_{\theta}(\vec s^k) < 0 \Leftrightarrow  \exists\, \vec s \in \partial H_{r^k}(\vec x^k) \text{~such that~} \widetilde{L}_{\theta}(\vec s) < 0,
\]
where $\vec s^k$ is selected through the strategy proposed in Section~\ref{sec:subgradient}. 
\end{theorem}


\begin{proof}
The necessity is evident and we only need to prove the sufficiency.  Suppose that we can select a subgradient $\vec s\in \partial H_{r^k}(\vec x^k)$ satisfying $\widetilde{L}_{\theta}(\vec s)<0$, then according to Lemma~\ref{prop:rs}, there exists an $t^*\in\{1,2,\ldots,n\}$ satisfying
\begin{equation}
	\left\{
	\begin{aligned}
		&\frac{b_{t^*}}{e}\leq s_{t^*}<0, &\text{if }t^*\in S^+(\vec x^k),\\
		&\frac{b_{t^*}}{e}\geq s_{t^*}>0, &\text{if }t^*\in S^-(\vec x^k),\\
		&	|b_{t^*}^\prime|\geq |l_{t^*}|>0, &\text{if }t^*\in S^<(\vec x^k) \text{ and } |b_{t^*}^\prime|>0,\\
		&\frac{x_{t^*}b_{t^*}}{e}= x_{t^*}s_{t^*}<0, & \text{if }t^*\in S^<(\vec x^k) \text{ and } |b_{t^*}^\prime|=0,
	\end{aligned}
	\right.
\end{equation}
and thus $V_{\vec b}^\prime$ in Eq.~\eqref{eq:wanted_ub1} is nonempty. Consequently, the selected vertex $i^*$ satisfies  $s^{k}_{i^*}=b_{i^*}/e$ with 
\begin{equation}
	\left\{
	\begin{aligned}
		&x_{i^*}s^{k}_{i^*}<0, &\text{if }i^*\in S^{\pm}(\vec x^k),\\
		&	|s^{k}_{i^*}|+\frac{(\theta-1)d_{i^*}}{e}>0, &\text{if }i^*\in S^<(\vec x^k) \text{ and } |b_{i^*}^\prime|>0,\\
		&	x_{i^*}s^{k}_{i^*}<0, &\text{if }i^*\in S^{<}(\vec x^k)\text{ and } |b_{i^*}^\prime|=0,
	\end{aligned}
	\right.
\end{equation}
which directly yields $\widetilde{L}_{\theta}(\vec s^{k})<0$ through Lemma~\ref{prop:rs}.
\end{proof}

The property of strict descent is guaranteed by  Lemma~\ref{prop:rs} and plays an important role in the following local convergence. 

\begin{theorem}[$\mathbf{SIP}_{\theta}$: local convergence]
\label{thm:conver_3theta}
Suppose the sequences $\{\vec x^k\}$ and $\{r^k\}$ are generated by $\mathbf{SIP}_{\theta}$ with $\theta\in(0,1]$ from any {nonconstant} initial point $\vec x^0\in \mathbb{R}^n\setminus \{\vec0\}$. Then there must exist $N\in\mathbb{Z}^+$ and $r^*\in\mathbb{R}$ such that, for any $k > N$, 
$r^k=r^*$ and $\vec x^{k+1}$ are {nonconstant} local minimizers in $U(\vec x^{k+1})$.
\end{theorem}

\begin{proof}
First, we prove that the sequence $\{r^k\}$ takes finite values. It can be derived from Eq.~\eqref{eq:teranry_vec} that $\vec x^{k+1}/\|\vec x^{k+1}\|_{\infty}\in\{-1,0,1\}^n$ holds and $\{r^k\}$ takes a limited number of values. Moreover, Lemma~\ref{prop:rs} and Theorem~\ref{prop:S123_theta} indicate that there exists $N\in\mathbb{Z}^+$ and $r^*\in\mathbb{R}$ such that  for any $k>N$, satisfying 
\begin{align}
	r^0>r^1>\cdots>r^{N}&>r^{N+1}=r^{N+2}=\cdots=r^*,\label{eq:r_list}
\end{align}
and thus we have: 
\begin{enumerate}
	\item When $k\le N$,  $\widetilde{L}_{\theta}(\vec s^k)<0$ (see Eq.~\eqref{eq:sip_theta_thm}) holds and $\vec x^{k+1}$ is nonconstant
	because any constant vector $\vec y$ leads to $L_\theta(\vec y, \vec s^k)\geq 0$; 
	\item When $k > N$, $\forall\, \vec s\in\partial H_{r^k}(\vec x^{k})$ and $\vec l^k$ (see Eq.~\eqref{eq:vec_l}), Eq.~\eqref{eq:exact_x} holds
	and thus a nonconstant $\vec x^{k+1}$ can be selected through the procedure proposed in Section~\ref{sec:exact_prob}. 
\end{enumerate}

Then, we prove $\vec x^{k+1}$ is a local minimizer of $T(\cdot)$ over $U(\vec x^{k+1})$ when $k \ge N$. 
According to Eq.~\eqref{eq:teranry_vec}, we have
\begin{equation*}
	\begin{cases}
		y_i>\frac{1}{2}\|\vec y\|_{\infty},&\text{if }i\in S^+(\vec x^{k+1}),\\
		y_i<-\frac{1}{2}\|\vec y\|_{\infty},&\text{if }i\in S^-(\vec x^{k+1}),\\ 		
		|y_i|<\frac{1}{2}\|\vec y\|_{\infty},&\text{if }i\in S^<(\vec x^{k+1}),
	\end{cases} \quad \forall\,\vec y\in U(\vec x^{k+1}),
\end{equation*}
and thus
\begin{equation*}
	\partial I^+(\vec y) \subseteq \partial I^+(\vec x^{k+1}), \;\; \partial N(\vec y)\subseteq \partial N(\vec x^{k+1}) \Rightarrow \partial H_{r^*}\vec (\vec y)\subseteq \partial H_{r^*}(\vec x^{k+1}),	
\end{equation*}
implying that 
$\forall\,\vec s^*\in\partial H_{r^*}\vec (\vec y)$, we have
\begin{equation*}
	\begin{aligned}
		T(\vec y)-T(\vec x^{k+1})
		=&\frac{e}{N(\vec y)}\left(\theta\|\vec y\|_{\infty}+\frac{1-\theta}{e}\|\vec y\|_{1,d}-H_{r^*}(\vec y)\right)\\
		=&\frac{e}{N(\vec y)}\cdot L_{\theta}(\vec y,\vec s^*)
		\geq\frac{e}{N(\vec y)}\cdot \widetilde{L}_{\theta}(\vec s^*)=0,
	\end{aligned}
\end{equation*}
where we have used $T(\vec x^{k+1})=r^*$ in the first equality and $\vec s\in\partial H_{r^*}(\vec x^{k+1})$ in the last equality. 
\end{proof}

It is precisely because of the strict decent guarantee in Theorem~\ref{prop:S123_theta} and the local convergence in Theorems~\ref{thm:conver_3theta} that we are able to set the condition for exiting $\mathbf{SIP_{\theta}}$ iteration \eqref{iter2-cheeger_2} to be the objective function $T(\vec x)$ is unchanged (see Fig.~\ref{flowchart}), and to except that $\mathbf{SIP_{\theta}}$ has a fast convergence against the iteration steps (see Section~\ref{SG}). Furthermore, we can prove that $\mathbf{SIP_{\theta}}$ with $\theta=1$, i.e.,  $\mathbf{SIP}$,  is also locally converged in a discrete sense.

\begin{theorem}[$\mathbf{SIP}$: discrete local convergence]
\label{thm:discrete}
Suppose $\mathbf{SIP}$ converges to $\vec x^*$.  Then $\vec x^*\in C_B$.
\end{theorem}

%

\begin{proof}
According to the solution selection procedure in Section~\ref{sec:exact_prob}, 
$\mathbf{SIP}$ must produce nonconstant binary valued solutions. Namely, $\vec x^*/\|\vec x^*\|_\infty \in \{-1, 1\}^n$.  
For simplicity, the superscript $*$ is neglected in the proof. Suppose the contrary $\vec x\notin C_B$ that there exists $i\in V$ satisfying 
\begin{equation}\label{eq:ass}
	B(R_i\vec x)<B(\vec x) \quad \Leftrightarrow \quad H_{r}(\vec x)-H_{r}(R_i\vec x)<0,\quad r=B(\vec x),
\end{equation}
because of $\|R_i\vec x\|_\infty=\|\vec x\|_\infty$. 		
Then, we have
\begin{align}
	I^+(\vec x)-I^+(R_i\vec x)=\pm\sum\limits_{j:\{j,i\}\in E} 2w_{ij}x_j,\,\,i\in S^\pm(\vec x), \label{eq:delta_iplus} \\
	\begin{aligned}
		N(\vec x)-N(R_i\vec x)
		=&-|\Delta-\mu_i x_i|+|\Delta+\mu_i x_i|
		\\
		=&\begin{cases}
			2\mu_ix_i,&\text{if }\Delta> \mu_i\|\vec x\|_{\infty},\\
			2\Delta x_i,&\text{if }\Delta \le \mu_i\|\vec x\|_{\infty},\\ 
			-2\mu_i x_i,&\text{if }\Delta< -\mu_i\|\vec x\|_{\infty},
		\end{cases}
	\end{aligned}
\end{align}
where $\Delta=\sum_{j=1}^n\mu_jx_j-\mu_ix_i$. Thus, it can be derived that 
\begin{equation}\label{eq:mmm}
	\begin{cases}
		1\in\median(\vec x)\text{ and }1\in\median(R_i\vec x),&\text{if }\Delta> \mu_i\|\vec x\|_{\infty},\\
		x_i\in\median(\vec x)\text{ and }-x_i\in\median(R_i\vec x),&\text{if }\Delta\leq \mu_i\|\vec x\|_{\infty},\\
		-1\in\median(\vec x)\text{ and }-1\in\median(R_i\vec x),&\text{if }\Delta<-\mu_i\|\vec x\|_{\infty}.
	\end{cases}
\end{equation}			
Let $u_i\in(\partial I^+(\vec x))_i$ with 
$$
u_i=\sum\limits_{j:\{j,i\}\in E} 2w_{ij}z_{ij},\,\,\,z_{ij} =x_j/\|\vec x\|_{\infty}\in \sgn(x_i+x_j),
$$
and $v_i\in(\partial N(\vec x))_i$ with 
$$
v_i =\begin{cases}
	\mu_i,&\text{if }\Delta> \mu_i\|\vec x\|_{\infty},\\
	\Delta,&\text{if }\Delta \le \mu_i\|\vec x\|_{\infty},\\ 		
	-\mu_i,&\text{if }\Delta< -\mu_i\|\vec x\|_{\infty}.
\end{cases}
$$
Accordingly, combining Eqs.~\eqref{eq:sub_s},  \eqref{eq:ass} and \eqref{eq:mmm} together 
leads to 
\begin{equation}
	\label{xisileq0}
	x_is_i=\frac{x_i u_i+rx_iv_i}{e}<0,
\end{equation}
which contradicts Lemma~\ref{prop:rs}.
\end{proof}



Finally, we would like to point out that either $\mathbf{IP}$ in \cite{HeinBuhler2010,Chang2015} or $\mathbf{SIP}_{\theta}$ with $\theta\in(0, 1)$ is unable to guarantee a discrete local minimum
as illustrated by Examples~\ref{ex:ip} and \ref{ex:sip_theta} below. 

\begin{example}
\label{ex:ip}
Let $G=(V,E)$ be the standard square graph of four vertices, i.e., $V=\{1,2,3,4\}$ and $E=\{\{1,2\},\{2,3\},\{3,4\},\{4,1\}\}$. Then $I(\vec x)=|x_1-x_2|+|x_2-x_3|+|x_3-x_4|+|x_4-x_1|$ and $N(\vec x)=\min_{c\in\mathbb{R}}\sum_{i=1}^42|x_i-c|$ with $\mu$ being the degree weights. 
We start $\mathbf{IP}$ in Eq.~\eqref{iter1-cheeger} with $\vec x^0=\frac{1}{4}(1,1,-1,1)$, $r^0=1$ and $\vec v^0=(\frac{2}{3},\frac{2}{3},-2,\frac{2}{3})$,
and find that the inner subproblem \eqref{eq:twostep_x2} takes zero as its minimum value. That is, $\mathbf{IP}$ cannot decrease the value of $I(\vec x)/N(\vec x)$ in Eq.~\eqref{eq:twostep_r2} and thus terminates with $\vec x^1=\vec x^0$. On the other hand, let $\vec y=R_4\vec x^1$, thus $B(\vec y)=\frac12<B(\vec x^1)=1$, indicating $\vec x^1\notin C_B$. 
\end{example}

\begin{example}
\label{ex:sip_theta}
Let $G=(V,E)$ be the standard path graph of six vertices, i.e., $V=\{1,2,3,4,5,6\}$ and $E=\{\{1,2\},\{2,3\},\{3,4\},\{4,5\},\{5,6\}\}$. Then $I^+(\vec x)=|x_1+x_2|+|x_2+x_3|+|x_3+x_4|+|x_4+x_5|+|x_5+x_6|$ and $N(\vec x)=\min_{c\in\mathbb{R}}\sum_{i=1}^6\mu_i|x_i-c|$ with $\mu$ being degree wights, i.e. $\mu=(1,2,2,2,2,1)$. 
We start $\mathbf{SIP}_{\theta}$ in Eq.~\eqref{iter2-cheeger_2} with $\theta=0.2$, $\vec x^0=\frac{1}{3}(0,1,0,1,1,0)$ and $r^0=0.2$, 
and calculate $\vec b^0=(0.8,2,1.6,2,2,0.8)$ (see Eq.~\eqref{u1}) which yields an empty $V_{\vec b^0}^\prime$ (see Eq.~\eqref{eq:wanted_ub1}). 
According to Theorem~\ref{prop:S123_theta}, 
$\mathbf{SIP}_{\theta}$ cannot improve the solution quality and then outputs 
$\vec x^1=\vec x^0$. However, $\vec x^1\notin C_T$ because 
$T(R_2\vec x^1)=\frac{2}{15}<T(\vec x^1)=\frac15$, where $C_T = \{\vec x\in\mathbb{R}^n \big| T(\vec y)\geq T(\vec x), \, \forall\,\vec y\in\{R_i\vec x: i\in\{1,\ldots,n\}\}\}$. 

\end{example}

\section{Numerical experiments}
\label{sec:experiments}	

This section is devoted into conducting performance evaluation of $\mathbf{SIP}$ and $\mathbf{SIP}$-$\mathbf{perturb}$
against $\mathbf{IP}$ and \texttt{Gurobi}, respectively, for Cheeger cut and Sparsest cut on a dataset comprising 30 graphs with positive weights from G-set\footnote{Downloaded from \href{https://web.stanford.edu/~yyye/yyye/Gset/}{https://web.stanford.edu/$\sim$yyye/yyye/Gset/}}. For initialization, we adopt the second smallest eigenvector of the normalized graph Laplacian, which is a common practice in the standard spectral clustering \cite{von2007tutorial}. All the algorithms are implemented using Matlab (r2019b) on the computing platform of 2*Intel Xeon E5-2650-v4 (2.2GHz, 30MB Cache, 9.6GT/s QPI Speed, 12 Cores, 24 Threads) with 128GB Memory.

\subsection{Comparisons between $\mathbf{SIP}$ and $\mathbf{IP}$}
\label{SIP-IP}

\begin{figure}[htbp]
\centering
\subfigure[Cheeger cut.]{\includegraphics[scale=0.14]{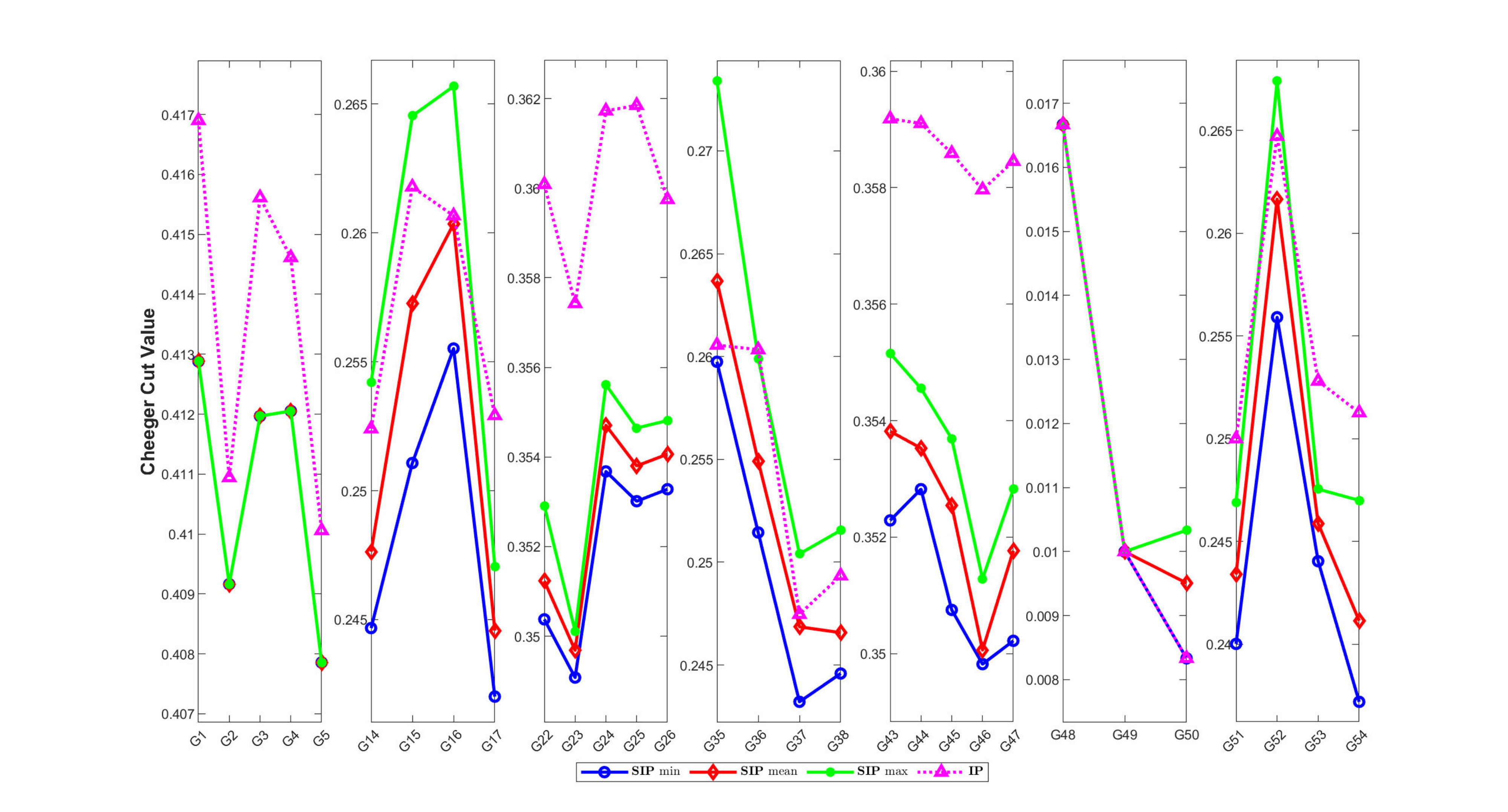}}
\subfigure[Sparsest cut.]{\includegraphics[scale=0.14]{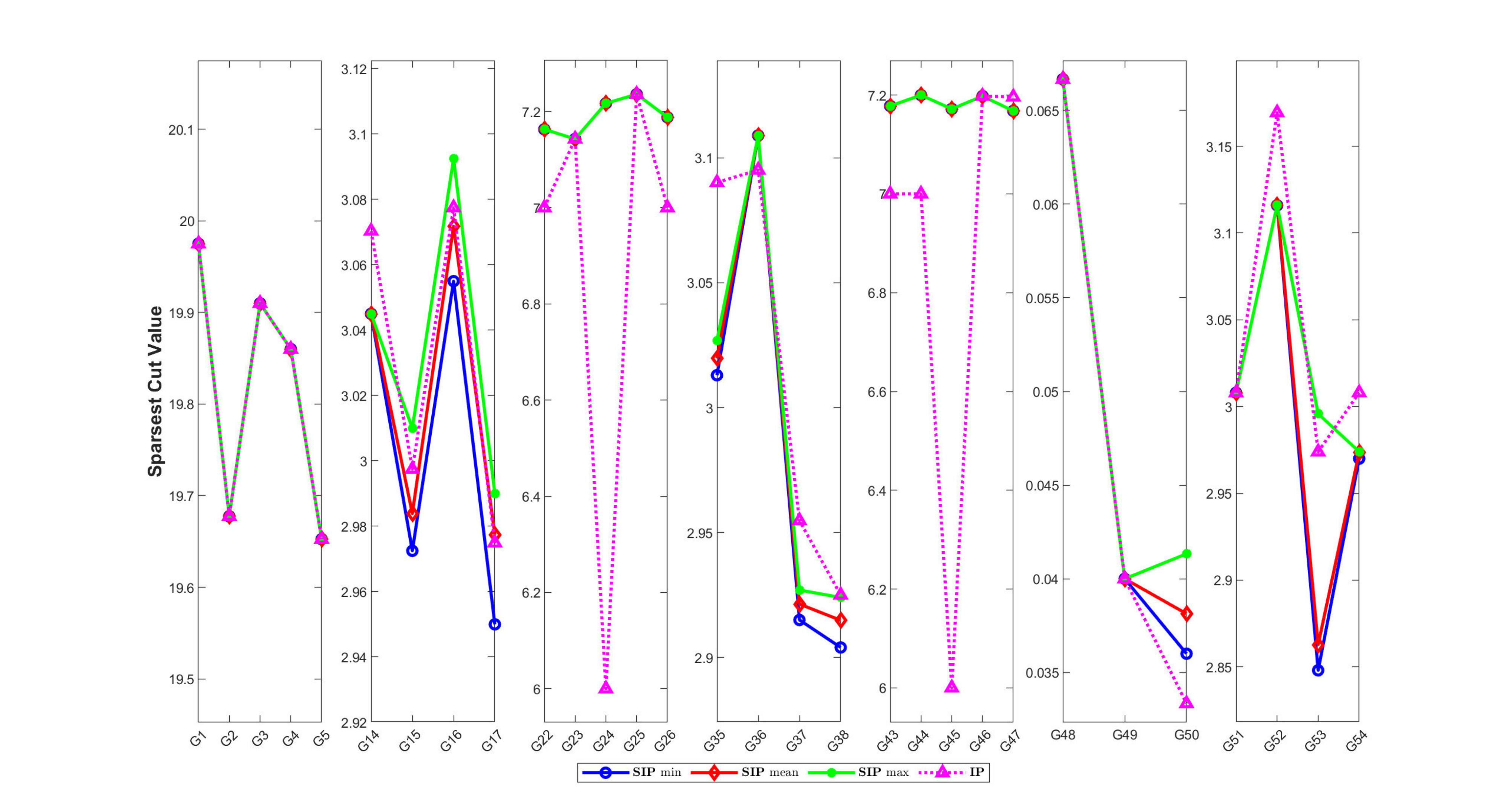}}
\caption{\small The minimum, mean, and maximum objective function values obtained by 40 runs of $\mathbf{SIP}$. The dotted lines present the corresponding values generated by $\mathbf{IP}$.}
\label{fig::compare}
\end{figure}

\begin{table}[htbp]
\centering
\caption{\small The average wall-clock time in seconds obtained by {$\mathbf{SIP}$} and $\mathbf{IP}$ for Cheeger cut and Sparsest cut on G-set. It is observed that {$\mathbf{SIP}$} exhibits faster execution time in comparison to $\mathbf{IP}$. Specifically, the running time cost of $\mathbf{IP}$ is at least 13.19 times that of {$\mathbf{SIP}$} (see G52 for Sparsest cut).}
\small     
\renewcommand{\arraystretch}{1.0}
\setlength{\tabcolsep}{2.1mm}{
	\begin{tabular}{|c|c|c|cc|cc|}
		\hline
		\multirow{2}{*}{instance} & \multirow{2}{*}{$|V|$} & \multirow{2}{*}{$|E|$} & \multicolumn{2}{c|}{Cheeger cut} & \multicolumn{2}{c|}{Sparsest cut} \\ \cline{4-7} 
		&                        &                        & {$\mathbf{SIP}$}           & $\mathbf{IP}$         & {$\mathbf{SIP}$}           & $\mathbf{IP}$         \\ \hline
		G1                        & 800                    & 19176                  & 0.3772                        & 5.7441                       & 0.3691                        & 6.5298                       \\
		G2                        & 800                    & 19176                  & 0.3429                        & 5.3790                       & 0.3742                        & 6.0206                       \\
		G3                        & 800                    & 19176                  & 0.3612                        & 5.4288                       & 0.3732                        & 5.6687                       \\
		G4                        & 800                    & 19176                  & 0.3449                        & 6.0036                       & 0.3737                        & 5.5978                       \\
		G5                        & 800                    & 19176                  & 0.3588                        & 6.6240                       & 0.3547                        & 5.5024                       \\
		G14                       & 800                    & 4694                   & 0.1433                        & 4.6874                       & 0.0961                        & 1.3998                       \\
		G15                       & 800                    & 4661                   & 0.1507                        & 3.1754                       & 0.1019                        & 2.9464                       \\
		G16                       & 800                    & 4672                   & 0.1551                        & 5.2827                       & 0.1079                        & 2.9248                       \\
		G17                       & 800                    & 4667                   & 0.1341                        & 2.4917                       & 0.0967                        & 3.7529                       \\
		G22                       & 2000                   & 19990                  & 0.4486                        & 16.7857                      & 0.4019                        & 14.2918                      \\
		G23                       & 2000                   & 19990                  & 0.4469                        & 20.2874                      & 0.3918                        & 10.7292                      \\
		G24                       & 2000                   & 19990                  & 0.4347                        & 17.2368                      & 0.3630                        & 13.4280                      \\
		G25                       & 2000                   & 19990                  & 0.5343                        & 15.2183                      & 0.3938                        & 10.8652                      \\
		G26                       & 2000                   & 19990                  & 0.4570                        & 15.2113                      & 0.3929                        & 14.1498                      \\
		G35                       & 2000                   & 11778                  & 0.4496                        & 15.0911                      & 0.3102                        & 8.0580                       \\
		G36                       & 2000                   & 11766                  & 0.4009                        & 9.2911                       & 0.2469                        & 7.8442                       \\
		G37                       & 2000                   & 11785                  & 0.3929                        & 11.7091                      & 0.2628                        & 10.1999                      \\
		G38                       & 2000                   & 11779                  & 0.3927                        & 20.3054                      & 0.2703                        & 9.8490                       \\
		G43                       & 1000                   & 9990                   & 0.1992                        & 5.3609                       & 0.1926                        & 3.9212                       \\
		G44                       & 1000                   & 9990                   & 0.2138                        & 5.5295                       & 0.1698                        & 4.0186                       \\
		G45                       & 1000                   & 9990                   & 0.2143                        & 5.2037                       & 0.1824                        & 3.9411                       \\
		G46                       & 1000                   & 9990                   & 0.2155                        & 5.3510                       & 0.1957                        & 3.0092                       \\
		G47                       & 1000                   & 9990                   & 0.2059                        & 7.3416                       & 0.1993                        & 3.0699                       \\
		G48                       & 3000                   & 6000                   & 0.1210                        & 2.0729                       & 0.1223                        & 1.7856                       \\
		G49                       & 3000                   & 6000                   & 0.1245                        & 2.1380                       & 0.1172                        & 1.8393                       \\
		G50                       & 3000                   & 6000                   & 0.1488                        & 2.1195                       & 0.1337                        & 1.7826                       \\
		G51                       & 1000                   & 5909                   & 0.1747                        & 8.3686                       & 0.1214                        & 2.6072                       \\
		G52                       & 1000                   & 5916                   & 0.1833                        & 4.3485                       & 0.1318                        & 1.7393                       \\
		G53                       & 1000                   & 5914                   & 0.1826                        & 5.4845                       & 0.1708                        & 6.1148                       \\
		G54                       & 1000                   & 5916                   & 0.1876                        & 7.1092                       & 0.1263                        & 2.6539                      \\ \hline
	\end{tabular}
}
\label{tab:time}
\end{table}

We rerun $\mathbf{SIP}$ in Eq.~\eqref{iter1-cheeger_2} and $\mathbf{IP}$ in Eq.~\eqref{iter1-cheeger} 40 times.
Each run of $\mathbf{SIP}$ terminates when $B(\vec x)$ in Eq.~\eqref{eq:twostep_r2_2} fails to decrease. For $\mathbf{IP}$, we use the \texttt{MOSEK} solver with ``\texttt{cvx\_precision high}" in CVX to solve the inner subproblem \eqref{eq:twostep_x2}, 
and use the resulting minimum value as the termination indicator. That is, when it reaches 0 $\mathbf{IP}$ stops \cite{HeinBuhler2010}.  Notably, the subgradient selection \eqref{eq:twostep_s2} utilized in $\mathbf{IP}$ follows the approach outlined in \cite{Chang2015} and remains deterministic.

Fig.~\ref{fig::compare} plots the minimum, mean and maximum objective function values of Cheeger cut and Sparsest cut obtained in 40 runs of $\mathbf{SIP}$ and $\mathbf{IP}$, accompanied by the average runtime in Table~\ref{tab:time}. It can be readily seen there that 
the numerical solutions generated by $\mathbf{IP}$ are identical (see the pink dotted lines in Fig.~\ref{fig::compare}), thereby implying that there is no randomness in its implementation.  Fig.~\ref{fig::compare} reveals that $\mathbf{SIP}$ outperforms $\mathbf{IP}$ on the majority of the graph instances for both Cheeger cut and Sparsest cut. Specifically, $\mathbf{SIP}$ achieves superior approximations for Cheeger cut across all graphs (see the blue solid lines v.s. the pink dotted lines in (a)), and better approximations for Sparsest cut than $\mathbf{IP}$ except on G22, G24, G26, G36, G43, G44 G45 and G50.

In terms of time complexity, $\mathbf{SIP}$ primarily requires a simple sort operation in the first subproblem~\eqref{eq:twostep_x2_2} (see Lemma~\ref{Thm:exact_solution}), with a worst-case computational cost of $\mathcal{O}(n^2)$, while the \texttt{MOSEK} solver required by $\mathbf{IP}$ employs a primal-dual interior-point algorithm with a worst-case cost of $\mathcal{O}(n^{3.5})$. In practice, $\mathbf{SIP}$ indeed exhibits significantly faster compared to $\mathbf{IP}$ (see Table~\ref{tab:time}). On average, the time required for each graph by $\mathbf{IP}$ is 28.99 times greater than that of $\mathbf{SIP}$ for Cheeger cut, and 24.66 times greater for Sparsest cut. Finally, we would like to emphasize that this superior computational efficiency is inherited from the closed-form of the inner subproblem solution in Lemma~\ref{Thm:exact_solution}.


\subsection{Comparisons between $\mathbf{SIP}$-$\mathbf{perturb}$ and \texttt{Gurobi}}
\label{SG}

A more detailed numerical investigation reveals that $\mathbf{SIP}$ exhibits a rapid decline in the objective function $B(\vec x)$, which is also implied early by the strict decent guarantee in Theorem~\ref{prop:S123_theta} and the local convergence in Theorems~\ref{thm:conver_3theta} and \ref{thm:discrete}. Fig~\ref{fig::si1} displays a histogram of the relative errors 
\begin{equation}
\label{eq:relative1}
\frac{B(\vec x^6)-B(\vec x^*)}{B(\vec x^*)}
\end{equation}
for all $40\times 30=1200$ runs, where $\vec x^6$ gives the numerical solution at the 6-th iteration and $\vec x^*$ the corresponding local optimum of $\mathbf{SIP}$. We are able to observe there that, (1) for Cheeger cut, all runs achieve relative errors within $6.71\%$, and $95\%$ of them are within $3.91\%$; (2) for Sparsest cut, all runs achieve relative errors within $8.27\%$, and $95\%$ of them are within $3.02\%$. 
Therefore, this fast convergence of $\mathbf{SIP}$ combined with its superior computational efficiency as already demonstrated in Section~\ref{SIP-IP} 
leaves us plenty of room for improvement by $\mathbf{SIP}$-$\mathbf{perturb}$ (see Fig.~\ref{flowchart}).
When $\mathbf{SIP}$ is not capable of decreasing the objective function $B(\vec x)$,
$\mathbf{SIP}$-$\mathbf{perturb}$ has recourse to $\mathbf{SIP}_{\theta}$ \eqref{iter2-cheeger_2} that has been developed within exactly the same iteration algorithm framework and shares the same objective function in the subgradient selection phase as $\mathbf{SIP}$.

\begin{figure}[htbp]
\centering
\subfigure[Cheeger cut.]{\includegraphics[scale=0.19]{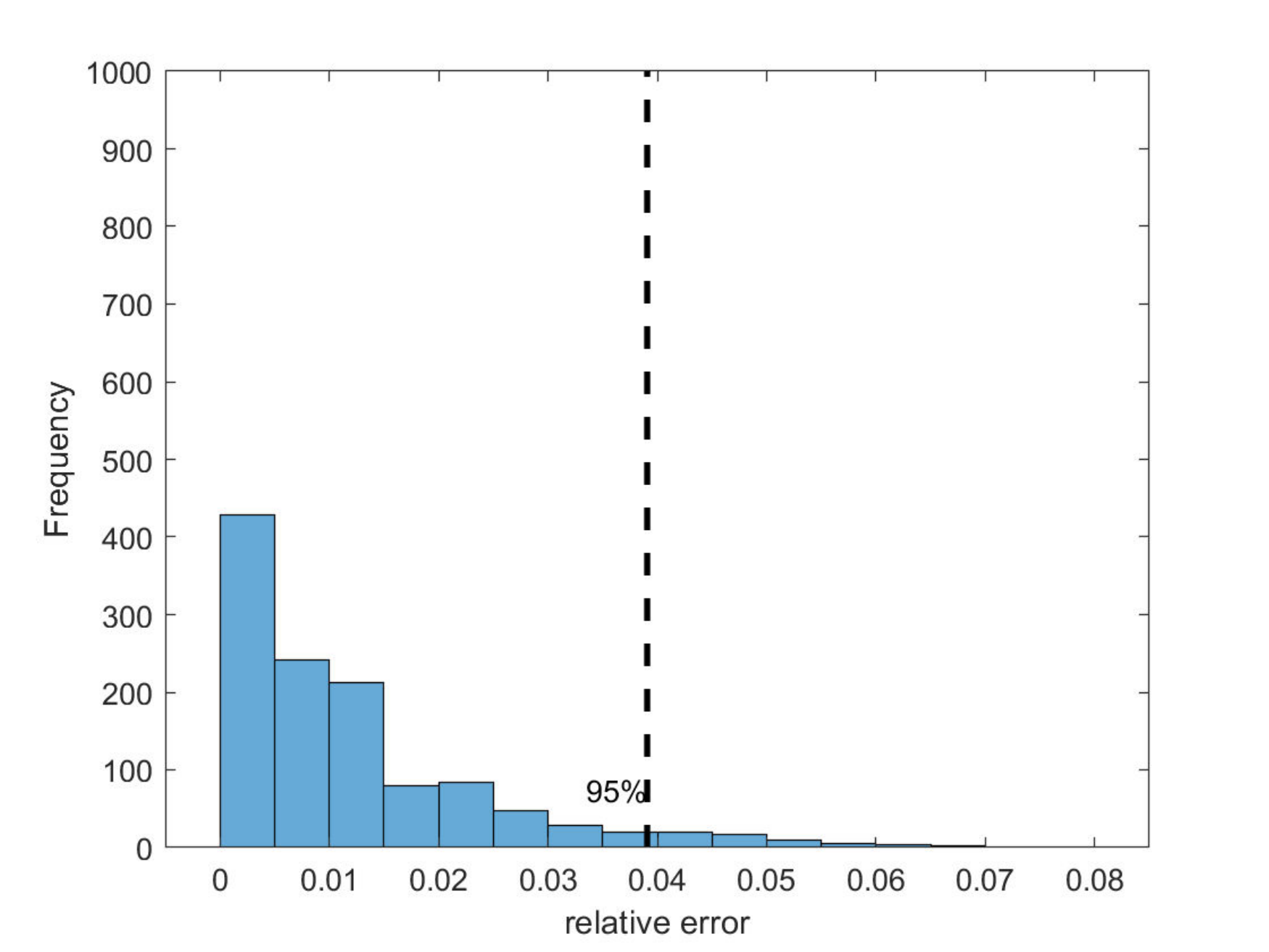}}
\subfigure[Sparsest cut.]{\includegraphics[scale=0.19]{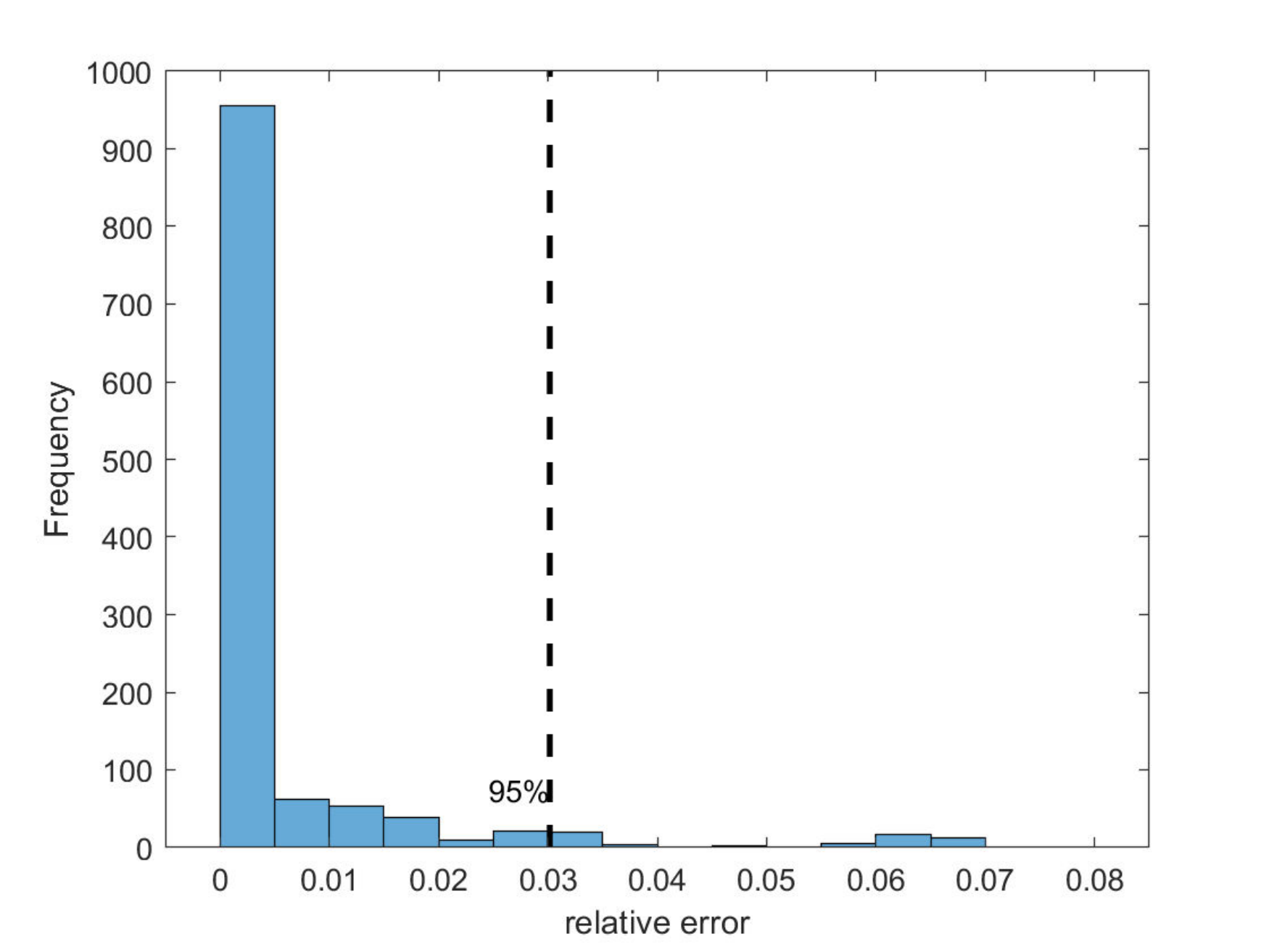}}
\caption{\small Fast convergence of $\mathbf{SIP}$: Histograms of the relative errors (see Eq.~\eqref{eq:relative1}) for all $40\times 30$ runs.} \label{fig::si1}
\end{figure}

In order to fully exploit the fast convergence implied by the strict decent guarantee in Theorem~\ref{prop:S123_theta} and the local convergence in Theorems~\ref{thm:conver_3theta} as well as the superior computational efficiency inherited from the closed-form solution to the inner subproblem in Lemma~\ref{Thm:exact_solution}, $\mathbf{SIP}$-$\mathbf{perturb}$ adopts $\mathbf{SIP}_{\theta}$ as a perturbation of $\mathbf{SIP}$.
To this end, we need to analyze how to adjust the parameter $\theta$ at each time step in the first place. Considering several typical regular graphs, such as the Petersen graph (denoted as $P_{10}$), the circle graph (e.g. $C_{10}$) and the complete graph (e.g. $K_{10}$), Fig.~\ref{fig1} plots their $h_{\theta}(G)$ as $\theta$ increases from $0$ to $1$. It can be easily seen there that $h_{\theta}(G)$ monotonically increases from the trivial value $0$ at $\theta=0$ to $h(G)$ after a certain threshold. Moreover, Fig.~\ref{fig:petersen} displays the optimal solutions corresponding to $h_{0}(G)$, $h_{\frac12}(G)$ and  $h_{1}(G)$ (i.e., $h(G)$)
on the Petersen graph and clearly demonstrates an increase in $h_{\theta}(G)$ but a decrease in $|(V_1\cup V_2)^c|$.
Therefore, we know that the parameter $\theta$ controls the portion of the un-partitioned vertices: A smaller $\theta$ indicates a larger number of un-partitioned vertices. Accordingly, we propose to randomly select $\theta$ from a uniform distribution on the interval $[0.3,0.8]$
before each call for $\mathbf{SIP}_{\theta}$, given the low edge density across all graph instances as shown in the first three columns of 
Table~\ref{tab:time}. It is noteworthy that the selection of $\theta$ here is not unique but  suitable, since the lower bound $0.3$ and the upper bound $0.8$ respectively sit around the midpoint and the endpoint of each ascending slope observed in Fig.~\ref{fig1}. A very small value of $\theta$ could potentially lead to an overabundance of vertices in the ``un-partitioned" part (see Fig.~\ref{fig:petersen}) and result in significant deviations of $h_{\theta}(G)$ from $h(G)$ (see Fig.~\ref{fig1}), consequently causing the undergoing ternary-valued solutions to significantly diverge from the current binary-valued ones. Conversely, a very large value of $\theta$ may be basically the same as the original $h(G)$ (see Fig.~\ref{fig1}) and thus may introduce an invalid perturbation. The adjustable $\theta$ enhances the diversification of $\mathbf{SIP}$-$\mathbf{perturb}$, ultimately yielding higher quality solutions (see Tables \ref{tab:cheeger} and \ref{tab:sparsest}).

%




\begin{figure}[htbp]
\centering
\includegraphics[scale=0.22]{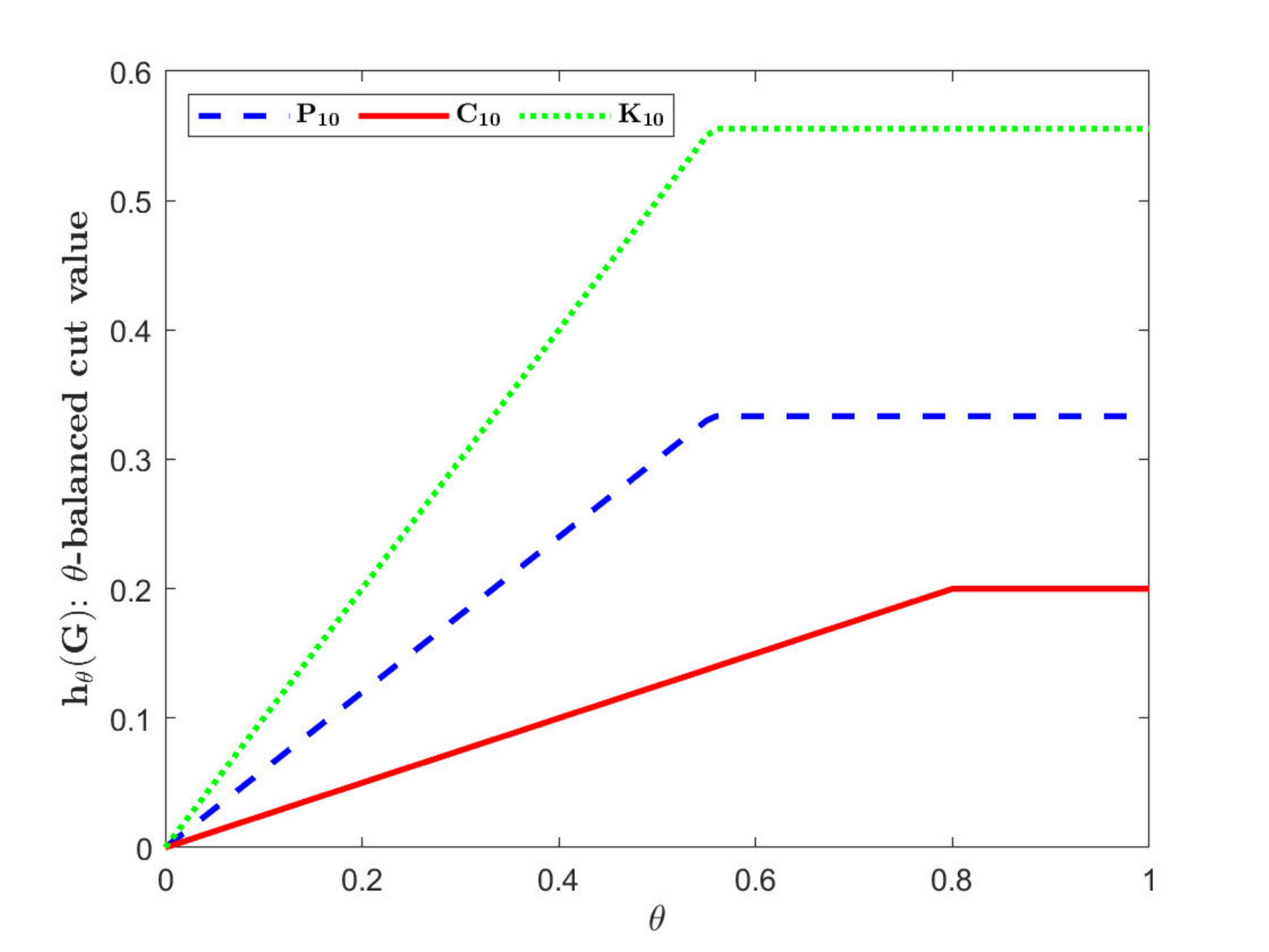}
\caption{\small The $\theta$-balanced cut values $h_{\theta}(G)$ on $P_{10}$, $C_{10}$ and $K_{10}$. Here $\theta$ increase from $0$ to $1$  with a fixed step of $0.01$ and the weight $\mu_i$ on any $i\in V$ is $d_i$,  the degree of the $i$-th vertex.}
\label{fig1}
\end{figure}

\begin{figure}[htbp] 
\centering
\includegraphics[scale=0.3]{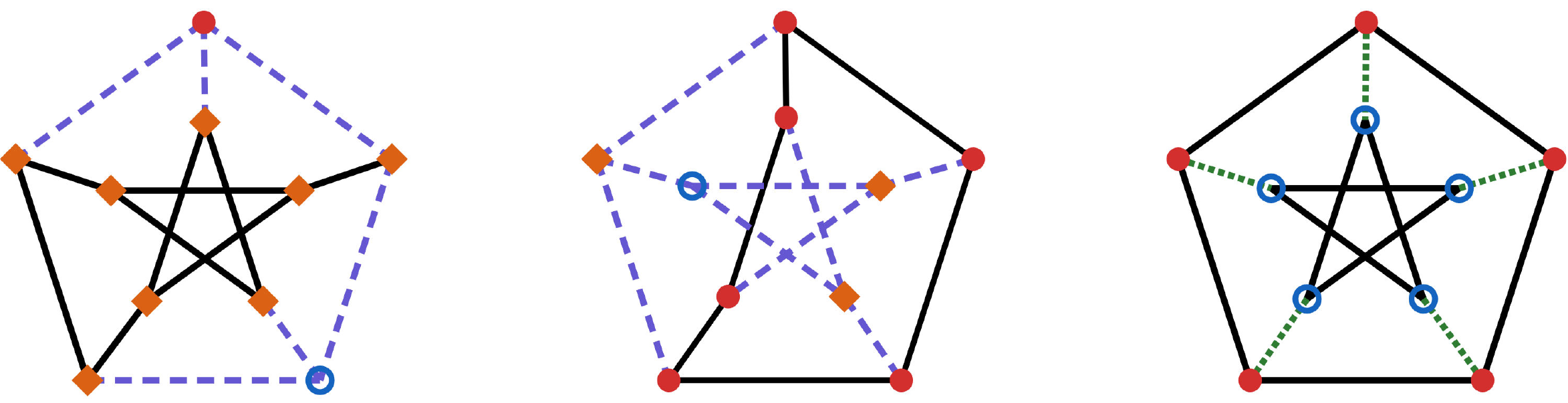}
\caption{\small $h_{0}(G)$ (left) {\it v.s.} $h_{\frac{1}{2}}(G)$ (center) {\it v.s.} $h_{1}(G)\equiv h(G)$ (right) on the Petersen graph: An example of the partitions $(V_1, V_2)\in\tc(V)$ for $V$ is displayed in the red bullets ($V_1$), blue circles ($V_2$) and orange squares ($(V_1\cup V_2)^c$). The dotted green lines represent $E(V_1,V_2)$, while $E(V_1\cup V_2,(V_1\cup V_2)^c)$ are shown in the dashed purple lines. We have $|(V_1\cup V_2)^c|=8$, $3$ and $0$ for $\theta=0$, $\frac{1}{2}$ and $1$, respectively, while the corresponding values of $h_{\theta}(G)$ are $0$, $\frac{3}{10}$ and $\frac{5}{15}$. }
\label{fig:petersen}
\end{figure}


To demonstrate the effect of $\mathbf{SIP}_{\theta}$ in improving the solution quality, 
we first perform 40 repeated runs of $\mathbf{SIP}$-$\mathbf{perturb}$ with $T_{\theta}=1$ for each graph, where $T_\theta$ gives the number of calling $\mathbf{SIP}_{\theta}$ (see Fig.~\ref{flowchart}). Here $T_{\theta}=1$ means, $\mathbf{SIP}$-$\mathbf{perturb}$ = $\mathbf{SIP} \rightarrow \mathbf{SIP}_{\theta} \rightarrow \mathbf{SIP}$. Namely, we call $\mathbf{SIP}$ twice for each run, and then calculate the relative gain
\begin{equation}
\label{eq:relative2}
\frac{B(\vec x^*)-B(\vec y^*)}{B(\vec x^*)}
\end{equation}
where $\vec x^*$ is the output of the first call of $\mathbf{SIP}$, and $\vec y^*$ the output of $\mathbf{SIP}$-$\mathbf{perturb}$.
A histogram of the relative gains is displayed in Fig.~\ref{fig::si2}. It is intuitive that $\mathbf{SIP}_{\theta}$ is an efficient perturbation: (1) For Cheeger cut, $79.67\%$ of the runs  are improved, $71.25\%$ have relative gain more than $0.5\%$, and the highest relative gain is $6.10\%$; (2) 
for Sparsest cut, $89.33\%$ of the runs are improved, $86.75\%$ have relative gain more than $0.5\%$, and the highest relative gain is $7.27\%$.

\begin{figure}[htbp]
\centering
\subfigure[Cheeger cut.]{\includegraphics[scale=0.19]{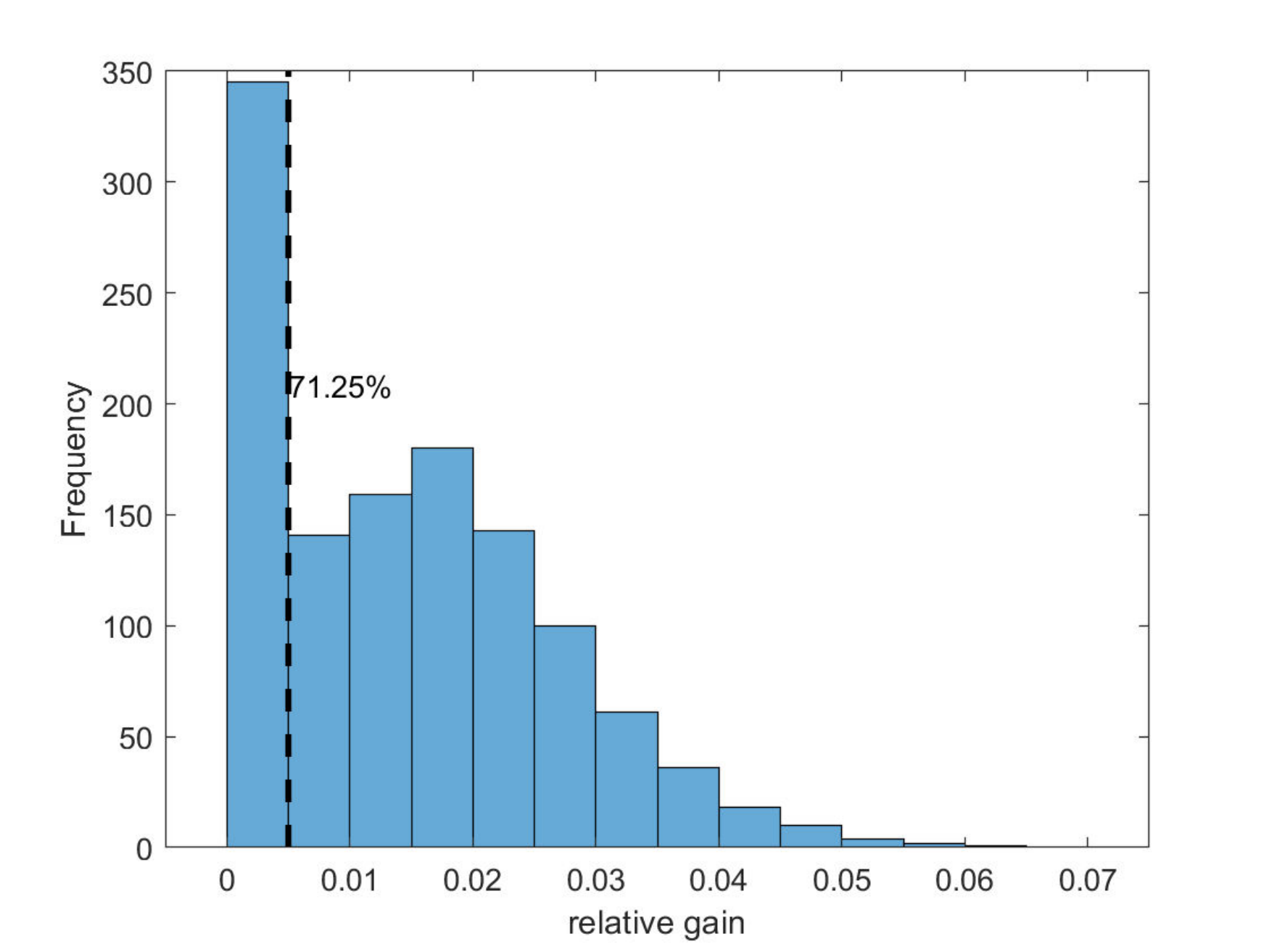}}
\subfigure[Sparsest cut.]{\includegraphics[scale=0.19]{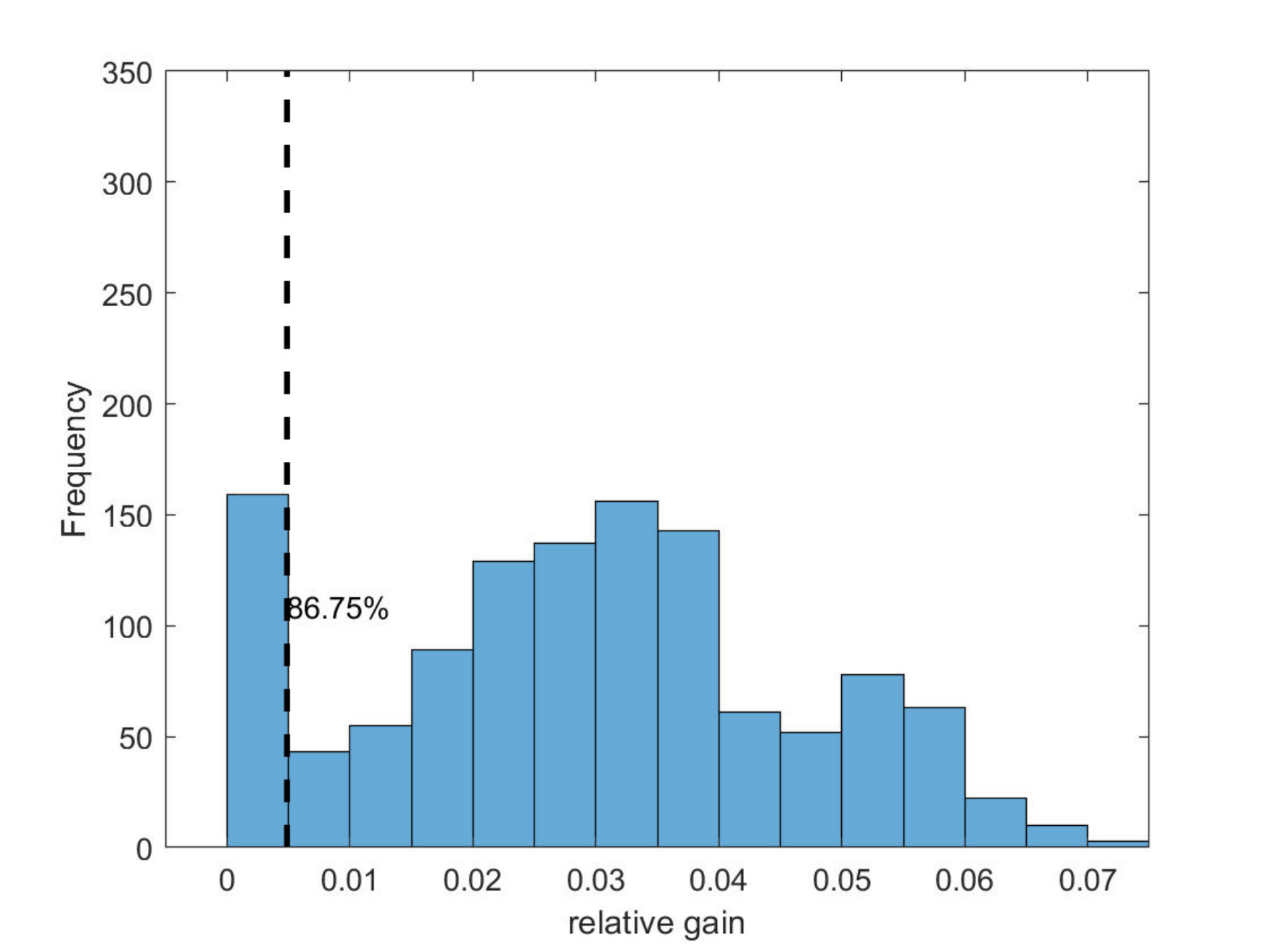}}
\caption{\small Improvement of solution quality by  $\mathbf{SIP}$-$\mathbf{perturb}$ with $T_{\theta}=1$ (see Fig.~\ref{flowchart}): Histograms of the relative gains (see Eq.~\eqref{eq:relative2}) for all $40\times 30$ runs.}
\label{fig::si2}
\end{figure}


Finally, we re-run $\mathbf{SIP}$-$\mathbf{perturb}$ with $T_{\theta}=200$ (see Fig. \ref{flowchart}) 40 times,
against the reference solutions generated using the well-known solver \texttt{Gurobi} (version 10) with a running time limit of 30 minutes for each graph instance. The default deterministic method in \texttt{Gurobi}, referred to as $\texttt{Gurobi}_1$, often yields less satisfactory results on most graphs compared to heuristic strategies. To improve it, we utilize the available heuristics in $\texttt{Gurobi}$ by setting "\texttt{NoRelHeurTime=timeLimit}" to obtain the method $\texttt{Gurobi}_2$. The minimum, mean, maximum objective function values obtained by $\mathbf{SIP}$-$\mathbf{perturb}$ as well as the average time in seconds are presented in Tables~\ref{tab:cheeger} and \ref{tab:sparsest}. It can be readily found there that $\mathbf{SIP}$-$\mathbf{perturb}$ consistently outperforms $\texttt{Gurobi}$ in providing better solutions on most graph instances. To be more specific, $\mathbf{SIP}$-$\mathbf{perturb}$ yields superior solutions compared to $\texttt{Gurobi}$ for all instances except G17, G35 and G53 for Cheeger cut, and except G24 and G45 for Sparsest cut. More importantly, as excepted, $\mathbf{SIP}$-$\mathbf{perturb}$ demonstrates notable runtime efficiency, requiring an average time of less than 75 seconds for each graph instance.

\begin{table}[htbp]
\centering
\caption{\small 
	Numerical results for Cheeger cut by $\mathbf{SIP}$-$\mathbf{perturb}$ with $T_{\theta}=200$ against \texttt{Gurobi} with a running time limit of 30 minutes for each graph. $\mathbf{SIP}$-$\mathbf{perturb}$ is rerun for 40 times, and the fourth, fifth, and sixth columns represent the minimum, average, and maximum objective function values, respectively. The seventh column displays the average time in seconds. The last two columns provide the results obtained by \texttt{Gurobi}$_1$ and \texttt{Gurobi}$_2$, 
	where $\texttt{Gurobi}_1$ denotes the default deterministic method in \texttt{Gurobi}, 
	and \texttt{Gurobi}$_2$ the heuristic method by setting "\texttt{NoRelHeurTime=timeLimit}". $\mathbf{SIP}$-$\mathbf{perturb}$ outperforms $\min(\texttt{Gurobi}_1, \texttt{Gurobi}_2)$ in terms of best Cheeger cut approximate values on all the graphs except for G17, G35 and G53.}
\small     
\renewcommand{\arraystretch}{1.0}
\setlength{\tabcolsep}{1.6mm}{
	\begin{tabular}{|c|c|c|cccc|c|c|}
		\hline
		\multirow{2}{*}{graph} & \multirow{2}{*}{$|V|$} & \multirow{2}{*}{$|E|$} & \multicolumn{4}{c|}{$\mathbf{SIP}$-$\mathbf{perturb}$}          & \multirow{2}{*}{\texttt{Gurobi}$_1$} & \multirow{2}{*}{\texttt{Gurobi}$_2$} \\ \cline{4-7}
		&                        &                        & min    & mean   & max    & time (seconds)   &                         &                      \\ \hline
		G1                        & 800                    & 19176                  & 0.3971 & 0.3992 & 0.4013 & 25.2925 & 0.4176                & 0.4031               \\
		G2                        & 800                    & 19176                  & 0.3979 & 0.3993 & 0.4012 & 24.1208 & 0.4132                & 0.4020               \\
		G3                        & 800                    & 19176                  & 0.3968 & 0.3990 & 0.4010 & 24.0715 & 0.4160                & 0.4036               \\
		G4                        & 800                    & 19176                  & 0.3972 & 0.3986 & 0.4000 & 28.3913 & 0.4295                & 0.3993               \\
		G5                        & 800                    & 19176                  & 0.3976 & 0.3986 & 0.3996 & 33.2754 & 0.4144                & 0.4037               \\
		G14                       & 800                    & 4694                   & 0.2367 & 0.2384 & 0.2410 & 10.9732 & 0.2474                & 0.2395               \\
		G15                       & 800                    & 4661                   & 0.2385 & 0.2412 & 0.2440 & 8.8301  & 0.2461                & 0.2396               \\
		G16                       & 800                    & 4672                   & 0.2323 & 0.2352 & 0.2391 & 9.9751  & 0.2456                & 0.2389               \\
		G17                       & 800                    & 4667                   & 0.2316 & 0.2349 & 0.2392 & 11.1397 & 0.2370                & 0.2295               \\
		G22                       & 2000                   & 19990                  & 0.3356 & 0.3388 & 0.3428 & 67.2310 & 0.3564                & 0.3404               \\
		G23                       & 2000                   & 19990                  & 0.3347 & 0.3360 & 0.3383 & 70.9517 & 0.3664                & 0.3435               \\
		G24                       & 2000                   & 19990                  & 0.3368 & 0.3395 & 0.3436 & 65.5364 & 0.3630                & 0.3411               \\
		G25                       & 2000                   & 19990                  & 0.3374 & 0.3405 & 0.3445 & 68.5895 & 0.3649                & 0.3431               \\
		G26                       & 2000                   & 19990                  & 0.3366 & 0.3387 & 0.3424 & 65.2483 & 0.3701                & 0.3421               \\
		G35                       & 2000                   & 11778                  & 0.2353 & 0.2388 & 0.2447 & 35.7837 & 0.2530                & 0.2348               \\
		G36                       & 2000                   & 11766                  & 0.2337 & 0.2364 & 0.2415 & 34.9933 & 0.2538                & 0.2381               \\
		G37                       & 2000                   & 11785                  & 0.2331 & 0.2346 & 0.2357 & 37.6673 & 0.2572                & 0.2369               \\
		G38                       & 2000                   & 11779                  & 0.2290 & 0.2306 & 0.2327 & 35.8389 & 0.2545                & 0.2322               \\
		G43                       & 1000                   & 9990                   & 0.3365 & 0.3384 & 0.3407 & 21.3137 & 0.3513                & 0.3417               \\
		G44                       & 1000                   & 9990                   & 0.3359 & 0.3391 & 0.3433 & 23.4580 & 0.3487                & 0.3449               \\
		G45                       & 1000                   & 9990                   & 0.3371 & 0.3389 & 0.3437 & 19.8516 & 0.3461                & 0.3435               \\
		G46                       & 1000                   & 9990                   & 0.3347 & 0.3368 & 0.3391 & 19.4566 & 0.3462                & 0.3409               \\
		G47                       & 1000                   & 9990                   & 0.3351 & 0.3371 & 0.3396 & 20.2874 & 0.3501                & 0.3421               \\
		G48                       & 3000                   & 6000                   & 0.0167 & 0.0175 & 0.0180 & 12.9493 & 0.0177                & 0.0200               \\
		G49                       & 3000                   & 6000                   & 0.0100 & 0.0100 & 0.0100 & 9.0241  & 0.0107                & 0.0100               \\
		G50                       & 3000                   & 6000                   & 0.0083 & 0.0083 & 0.0083 & 8.2793  & 0.0083                & 0.0083               \\
		G51                       & 1000                   & 5909                   & 0.2337 & 0.2351 & 0.2366 & 15.7767 & 0.2520                & 0.2376               \\
		G52                       & 1000                   & 5916                   & 0.2335 & 0.2385 & 0.2448 & 14.3792 & 0.2603                & 0.2396               \\
		G53                       & 1000                   & 5914                   & 0.2327 & 0.2370 & 0.2403 & 14.1083 & 0.2384                & 0.2310               \\
		G54                       & 1000                   & 5916                   & 0.2285 & 0.2301 & 0.2314 & 13.9340 & 0.2541                & 0.2390                     \\ \hline
	\end{tabular}
}
\label{tab:cheeger}
\end{table}

\begin{table}[htbp]
\centering
\caption{\small 
	Numerical results for Cheeger cut by $\mathbf{SIP}$-$\mathbf{perturb}$ with $T_{\theta}=200$ against \texttt{Gurobi} with a running time limit of 30 minutes for each graph. $\mathbf{SIP}$-$\mathbf{perturb}$ is rerun for 40 times, and the fourth, fifth, and sixth columns represent the minimum, average, and maximum objective function values, respectively. The seventh column displays the average time in seconds. The last two columns provide the results obtained by \texttt{Gurobi}$_1$ and \texttt{Gurobi}$_2$, 
	where $\texttt{Gurobi}_1$ denotes the default deterministic method in \texttt{Gurobi}, 
	and \texttt{Gurobi}$_2$ the heuristic method by setting "\texttt{NoRelHeurTime=timeLimit}". $\mathbf{SIP}$-$\mathbf{perturb}$ outperforms $\min(\texttt{Gurobi}_1, \texttt{Gurobi}_2)$ in terms of best Sparsest cut approximate values on all the graphs except for G24 and G45.}
\small     
\renewcommand{\arraystretch}{1.0}
\setlength{\tabcolsep}{1.6mm}{
	\begin{tabular}{|c|c|c|cccc|c|c|}
		\hline
		\multirow{2}{*}{graph} & \multirow{2}{*}{$|V|$} & \multirow{2}{*}{$|E|$} & \multicolumn{4}{c|}{$\mathbf{SIP}$-$\mathbf{perturb}$}          & \multirow{2}{*}{\texttt{Gurobi}$_1$} & \multirow{2}{*}{\texttt{Gurobi}$_2$} \\ \cline{4-7}
		&                        &                        & min    & mean   & max    & time (seconds)    &                         &                      \\ \hline
		G1                        & 800                    & 19176                  & 18.9850 & 19.0675 & 19.1100 & 8.5504  & 19.9225               & 19.2750              \\
		G2                        & 800                    & 19176                  & 18.9975 & 19.0624 & 19.1100 & 8.8842  & 20.1000               & 19.2350              \\
		G3                        & 800                    & 19176                  & 18.9900 & 19.0388 & 19.0975 & 8.5538  & 19.6775               & 19.2500              \\
		G4                        & 800                    & 19176                  & 19.0125 & 19.0292 & 19.0850 & 8.3018  & 19.4925               & 19.2500              \\
		G5                        & 800                    & 19176                  & 19.0050 & 19.0307 & 19.0550 & 8.2012  & 30.0000               & 19.3850              \\
		G14                       & 800                    & 4694                   & 2.7775  & 2.7827  & 2.7850  & 6.1621  & 2.9325                & 2.8500               \\
		G15                       & 800                    & 4661                   & 2.7750  & 2.7911  & 2.8050  & 6.5986  & 2.8725                & 2.7925               \\
		G16                       & 800                    & 4672                   & 2.6875  & 2.7539  & 2.7750  & 9.0342  & 2.9300                & 2.7775               \\
		G17                       & 800                    & 4667                   & 2.6575  & 2.6952  & 2.7625  & 9.6973  & 2.9075                & 2.6925               \\
		G22                       & 2000                   & 19990                  & 6.7040  & 6.7281  & 6.7410  & 18.0138 & 7.0000                & 6.7990               \\
		G23                       & 2000                   & 19990                  & 6.6870  & 6.6948  & 6.7120  & 13.6658 & 8.0000                & 6.8250               \\
		G24                       & 2000                   & 19990                  & 6.7110  & 6.7273  & 6.7460  & 13.9231 & 6.0000                & 6.8710               \\
		G25                       & 2000                   & 19990                  & 6.7140  & 6.7380  & 6.7680  & 18.9209 & 8.0000                & 6.8770               \\
		G26                       & 2000                   & 19990                  & 6.6970  & 6.7096  & 6.7350  & 14.3524 & 7.0000                & 6.8260               \\
		G35                       & 2000                   & 11778                  & 2.7280  & 2.7589  & 2.7870  & 62.4783 & 3.2290                & 2.7310               \\
		G36                       & 2000                   & 11766                  & 2.7460  & 2.7725  & 2.8040  & 21.1304 & 4.0000                & 2.7560               \\
		G37                       & 2000                   & 11785                  & 2.7400  & 2.7527  & 2.7600  & 28.0591 & 4.0000                & 2.7950               \\
		G38                       & 2000                   & 11779                  & 2.6760  & 2.6840  & 2.6990  & 74.6944 & 4.0000                & 2.7340               \\
		G43                       & 1000                   & 9990                   & 6.7060  & 6.7287  & 6.7560  & 7.2454  & 7.0000                & 6.7960               \\
		G44                       & 1000                   & 9990                   & 6.7220  & 6.7422  & 6.7780  & 7.6587  & 7.0000                & 6.8700               \\
		G45                       & 1000                   & 9990                   & 6.7060  & 6.7293  & 6.7560  & 7.9143  & 6.0000                & 6.9520               \\
		G46                       & 1000                   & 9990                   & 6.6860  & 6.6989  & 6.7140  & 8.2106  & 8.0000                & 6.9100               \\
		G47                       & 1000                   & 9990                   & 6.6920  & 6.7112  & 6.7300  & 7.5275  & 8.0000                & 6.8940               \\
		G48                       & 3000                   & 6000                   & 0.0667  & 0.0696  & 0.0720  & 4.6774  & 0.0933                & 0.0667               \\
		G49                       & 3000                   & 6000                   & 0.0400  & 0.0400  & 0.0400  & 3.2445  & 0.0426                & 0.0400               \\
		G50                       & 3000                   & 6000                   & 0.0333  & 0.0333  & 0.0333  & 2.9073  & 0.0360                & 0.0333               \\
		G51                       & 1000                   & 5909                   & 2.7500  & 2.7614  & 2.7820  & 11.2056 & 2.8980                & 2.8360               \\
		G52                       & 1000                   & 5916                   & 2.7400  & 2.7846  & 2.8500  & 20.7225 & 2.9120                & 2.8280               \\
		G53                       & 1000                   & 5914                   & 2.7200  & 2.7408  & 2.7680  & 43.3088 & 2.9960                & 2.7920               \\
		G54                       & 1000                   & 5916                   & 2.7100  & 2.7178  & 2.7260  & 9.9858  & 2.9420                & 2.8920                   \\ \hline
	\end{tabular}
}
\label{tab:sparsest}
\end{table}

\section{Conclusions}
\label{sec:conclusion}

Our motivation for developing the simple inverse power ($\mathbf{SIP}$) method is to address the limitations in the existing inverse power ($\mathbf{IP}$) method for approximating the balanced cut problem. The biggest advantage of
$\mathbf{SIP}$ over $\mathbf{IP}$ is that the inner subproblem of the former allows an explicit analytic solution while the latter does not. This inspires a boundary-detected subgradient selection, ensures the objective function values to strictly decrease during the iterations, and thus help $\mathbf{SIP}$ establish the local convergence, which $\mathbf{IP}$ lacks either. Meanwhile, we applied the same idea into solving a ternary valued $\theta$-balanced cut and obtained $\mathbf{SIP}_{\theta}$ which can be regarded as a perturbation of $\mathbf{SIP}$, thereby resulting into $\mathbf{SIP}$-$\mathbf{perturb}$, an efficient local breakout improvement of $\mathbf{SIP}$. We validated $\mathbf{SIP}$ and $\mathbf{SIP}$-$\mathbf{perturb}$ on G-set in terms of both computational cost and solution quality. 
For the future, 
it seems worthwhile to explore the potential of the Lov{\'a}sz extension in embedding combinatorial structures into continuous spaces and enabling an equivalent continuous formulation which may pave the way for a similar simple inverse power method. Furthermore, this paper also leaves some open questions, including: How can we characterize the difference between the classic equivalent formulation (see Eq.~\eqref{conti-prob:balance}) and the new one (see Eq.~\eqref{conti2-prob:balance})? Is there a theoretical approximation ratio for such simple inverse power method?

\end{document}